\documentclass[a4paper, 11pt]{article}
\usepackage[utf8]{inputenc}

\usepackage{amsthm}
\usepackage{amsmath}
\usepackage{amssymb}
\usepackage{amsfonts}
\usepackage{bm}
\usepackage{bbm}
\usepackage{ctable}
\usepackage{enumerate}
\usepackage{fullpage}
\usepackage{graphicx}
\usepackage{hyperref}
\usepackage{lipsum}
\usepackage{lineno}
\usepackage{multirow}
\usepackage{multicol}
\usepackage{psfrag}
\usepackage{titlesec}
\usepackage{xcolor}

\usepackage[ruled]{algorithm2e}
\usepackage[shortlabels]{enumitem}
\usepackage[square,numbers]{natbib}

\titleformat*{\section}{\Large\bfseries}
\titleformat*{\subsection}{\large\bfseries}
\titleformat*{\subsubsection}{\normalsize\bfseries}
\titleformat*{\paragraph}{\normalsize\bfseries}

\theoremstyle{plain}
\newtheorem{claim}{\protect\claimname}
\theoremstyle{plain}

\theoremstyle{plain}
\newtheorem{lemma}[claim]{\protect\lemmaname}
\theoremstyle{plain}
\newtheorem{theorem}[claim]{\protect\theoremname}
\theoremstyle{plain}

\theoremstyle{definition}

\theoremstyle{definition}

\theoremstyle{definition}

\providecommand{\claimname}{Claim}
\providecommand{\lemmaname}{Lemma}
\providecommand{\propositionname}{Proposition}
\providecommand{\theoremname}{Theorem}
\providecommand{\corollaryname}{Corollary} 
\providecommand{\definitionname}{Definition}
\providecommand{\assumptionname}{Assumption}
\providecommand{\factname}{Fact}

\DeclareMathOperator*{\subto}{s.t.}
\DeclareMathOperator*{\cproduct}{\operatorname*{\oplus}}

\newcommand{\OPT}{{\sf OPT}}
\newcommand{\eps}{\varepsilon}

\newcommand{\Tr}{\mathbf{Tr}}

\newcommand{\mm}[1]{{\boldsymbol{#1}}}
\newcommand{\bb}[1]{\mathbbm{#1}}
\newcommand{\cc}[1]{\mathcal{#1}}
\newcommand{\xt}[2]{{#1}^{(#2)}}

\newcommand{\eigenval}[2]{\lambda_{#1}({#2})}

\newcommand{\mmexp}{\mathbf{exp}}
\newcommand{\mmint}{\mathbf{int}}
\newcommand{\mmdiag}{\mathbf{diag}}

\newcommand{\bR}{\bb{R}}

\newcommand{\mma}{\mm{a}}
\newcommand{\mmb}{\mm{b}}
\newcommand{\mmc}{\mm{c}}

\newcommand{\mme}{\mm{e}}

\newcommand{\mmm}{\mm{m}}

\newcommand{\mmp}{\mm{p}}
\newcommand{\mmq}{\mm{q}}

\newcommand{\mmu}{\mm{u}}
\newcommand{\mmv}{\mm{v}}
\newcommand{\mmw}{\mm{w}}
\newcommand{\mmx}{\mm{x}}
\newcommand{\mmy}{\mm{y}}

\newcommand{\mmA}{\mm{A}}

\newcommand{\mmP}{\mm{P}}
\newcommand{\mmQ}{\mm{Q}}

\newcommand{\mmX}{\mm{X}}
\newcommand{\mmY}{\mm{Y}}

\newcommand{\mmgamma}{{\boldsymbol{\gamma}}}

\newcommand{\mmmu}{{\boldsymbol{\mu}}}
\newcommand{\mmxi}{{\boldsymbol{\xi}}}

\newcommand{\mmone}{\mm{1}}
\newcommand{\mmzero}{\mm{0}}

\newcommand{\R}{\mathbb{R}}

\newcommand{\T}{\mathcal{T}}

\usepackage{enumerate}
\usepackage{mdframed}
\mdfdefinestyle{MyFrame}{%
    linecolor=black,
    outerlinewidth=2pt,
    innertopmargin=4pt,
    innerbottommargin=4pt,
    innerrightmargin=4pt,
    innerleftmargin=4pt,
        leftmargin = 4pt,
        rightmargin = 4pt
        }

\def\bpm{\begin{pmatrix}}
\def\epm{\end{pmatrix}}
\def\bal{\begin{aligned}}
\def\eal{\end{aligned}}

\newcommand{\cj}{\mathcal{J}}

\newcommand{\bi}{\begin{itemize}}
\newcommand{\ei}{\end{itemize}}
\newcommand{\beq}{\begin{equation}}
\newcommand{\eeq}{\end{equation}}

\newcommand{\ben}{\begin{enumerate}}
\newcommand{\een}{\end{enumerate}}

\title {A Primal-Dual  
Framework for Symmetric Cone Programming}%

\author{
Jiaqi Zheng\thanks{Department of Computer Science, National University of Singapore}
\and 
Antonios Varvitsiotis\thanks{Engineering Systems and Design Pillar, Singapore University of Technology and Design}
\and
Tiow-Seng Tan\footnotemark[1]
\and 
Wayne Lin\footnotemark[2]
}

\begin{document}

\maketitle
\vspace{-1em}
\begin{abstract}
 In this paper, we introduce a  primal-dual algorithmic framework for solving Symmetric Cone Programs (SCPs),   a versatile optimization model that unifies and extends Linear, Second-Order Cone (SOCP), and Semidefinite Programming (SDP).  Our work  generalizes  the primal-dual framework for SDPs introduced by Arora and Kale \cite{AKcombinatorial},  leveraging  a recent extension of the Multiplicative Weights Update method (MWU) to symmetric cones. Going beyond  existing works, our framework can handle SOCPs and  mixed SCPs, exhibits nearly linear time complexity, and can  be effectively parallelized.  
To illustrate the efficacy of our framework, we employ it to develop approximation algorithms for two geometric optimization problems: the Smallest Enclosing Sphere problem and the Support Vector Machine problem. Our theoretical analyses demonstrate that the two algorithms compute approximate solutions 
in nearly linear running time 
and with  parallel depth  scaling  polylogarithmically with the input size.
We compare our algorithms against CGAL as well as interior point solvers applied to 
these problems. 
Experiments show that our algorithms are highly efficient when implemented on a CPU and  achieve substantial  speedups  when  parallelized on a GPU, 
allowing us  to solve large-scale instances of these problems.

\end{abstract}


\section{Introduction}\label{sec:intro}
A {\em linear conic program}  (LCP) is an optimization problem that minimizes a linear objective function over the intersection of a 
convex cone $\mathcal{K}$ and a finite number of halfspaces:
\begin{equation}
\label{scp}
\tag{LCP}
\begin{aligned}
    \max \quad & \mm{c} \bullet \mm{x} \\
    \subto \quad & \mm{a}_j \bullet \mm{x} \le b_j \quad \forall j \in [m]\\
    & \mm{x} \succeq_\cc{K} 0
\end{aligned}
\hspace{4em}
\begin{aligned}
    \min \quad & \mm{b}^T \mm{y}\\
    \subto \quad & \textstyle\sum_{j=1}^{m} \mm{a}_j y_j \succeq_{\cc{K}^*}  \mm{c}\\
    & \mmy \in \mathbb{R}^m_+
\end{aligned}
\end{equation} 
Here $\cc{K}^*$ denotes the dual cone of $\cc{K}$ and $\succeq_\cc{K}$ denotes the generalized inequality induced by the cone $\cc{K}$, i.e. $\mmx \succeq_\cc{K} \mmy $ iff $\mmx-\mmy\in \cc{K}$. 
Some of the most  important convex optimization models are instances of LCPs with respect to appropriate convex cones. For example, LCPs over the $d$-dimensional non-negative orthant~$\mathbb{R}^d_+$ correspond to {\em linear programs} 
{(LPs), and}
LCPs over the cone of $d\times d$ positive semidefinite (PSD) matrices $\cc{S}^d_+$  correspond to {\em semidefinite programs} (SDPs).  Additionally, LCPs related to second-order cone $\cc{Q}^{d+1}$
correspond to  {\em second-order cone programs} (SOCPs). These optimization models -- LPs, SDPs, and SOCPs -- hold a prominent place in the realm of convex optimization, finding applications in a wide array of domains.

A {\em symmetric cone} is a closed convex cone that is self-dual and homogeneous.
The non-negative orthant, the  cone of (real or Hermitian) PSD matrices, the second-order 
{cone, and Cartesian products thereof} are all examples of symmetric cones, and a {\em symmetric cone program} (SCP) is simply an LCP over a symmetric cone.
Consequently, SCPs provide a unifying framework for studying LPs, SDPs, and SOCPs.
From an algorithmic standpoint, SCPs admit self-concordant barrier functions, which facilitate the development and analysis of interior point methods (IPMs)~\cite{Fab,Nem}.

Many approximation algorithms have been introduced in the literature for LPs and SDPs based on the Multiplicative Weight Update (MWU) method over the simplex, e.g.~\cite{pst:plotkin1995fast,AHK}, and its matrix variant over the set of density matrices, e.g.~\cite{AKcombinatorial}. In this work, we focus on algorithms based on the MWU framework that are  primal-dual, in the sense that, at the end of execution, we obtain a pair of primal/dual feasible solutions with an optimality gap bounded by the desired error tolerance. An important instance of this is the primal-dual SDP framework by Arora and Kale for SDPs \cite{AKcombinatorial}. MWU-based algorithms have a running time that is nearly linear in the input size and are amenable to parallelization due to the nature of the multiplicative update.

Motivated by the success of such  methods for LPs and SDPs, in this work we try to 
understand:
\begin{center}
\parbox[c]{400pt}{\centering {\em  
Can we develop primal-dual MWU-based algorithms for symmetric cone programs that have  nearly linear time complexity and are effectively parallelizable? }}
\end{center}

In this work, we introduce the first primal-dual approximation framework for SCPs, which is detailed in Section~\ref{sec:algorithm}. 
Our framework is motivated by, and generalizes, the primal-dual framework for SDPs developed by Arora and Kale~\cite{AKcombinatorial}. To devise our framework, we rely on the recently introduced {\em symmetric cone multiplicative weight update} method (SCMWU)~\cite{scmwu} that is designed for online optimization over symmetric cones.

In contrast to existing  MWU-based algorithms designed solely for LPs or SDPs, our new framework can handle  SOCPs and, additionally, SCPs with mixed conic constraints, 
 e.g., problems featuring both linear constraints and positive-semidefinite constraints.
Compared to IPMs, which require solving systems of linear equations at each iteration,  our framework has significantly lower computational costs and can be effectively parallelized.  Additionally, we offer geometric insights into our primal-dual framework that are also applicable to the previous MWU-based frameworks for LPs~\cite{pst:plotkin1995fast} and SDPs~\cite{AKcombinatorial}.

As a practical demonstration of our framework we develop approximation algorithms for two computational geometry problems. The first one is  the {\em smallest enclosing sphere} (SES) problem {of finding the sphere of minimum radius that {encloses} an input set of spheres}, and the  second one is the {\em support vector machine} (SVM) problem {of finding a hyperplane that separates two sets of data with maximum margin}. 
Our framework allows various degrees of freedom that necessitate further specification for  each individual problem, which we provide  in Sections~\ref{sec:ses} and~\ref{sec:pd} respectively.
When applied to the SES problem, our framework yields a parallelizable approach (dual-only algorithm in this context) that computes a $(1 + \eps)$-approximate solution with a 
running time of $O(\frac{nd \log n}{\eps^2})$,
where $n$ represents the number of inputs and $d$ is the dimensionality. 
When applied to the SVM problem, our framework yields a parallelizable primal-dual algorithm that computes a $(1-\eps)$-approximate solution for SVM and a $(1 + \eps)$-approximate solution for its dual counterpart in $O(\frac{E nd\log n}{\eps^2})$ time,
where $n$ is the number of data points, $d$ is the dimensionality, and $E$ is a instance-specific parameter that measures its difficulty.

In addition to the theoretical analyses, our algorithms for SES and SVM have been implemented in both sequential (CPU) and parallel (GPU) settings. Extensive experiments have been conducted to compare them with CGAL and IPM solvers in Section~\ref{sec:experiment}. The results show that our algorithms are GPU-friendly and possess fairly large parallelism. For large-scale input instances, our GPU implementations outperform the commercial solvers.

{\bf Related work. }
Plotkin, Shmoys, and Tardos~\cite{pst:plotkin1995fast} introduced a framework based on the classical MWU method for computing approximate solutions to packing and covering LPs with a running time that is nearly linear in the input size.
The framework of Plotkin et al.~\cite{pst:plotkin1995fast} was later employed by Klein and Lu~\cite{KL96} to devise fast approximation algorithms for the {\sc Max Cut} problem, which was further extended by Arora et al.~\cite{AHK} for solving general SDPs.
These methods can be categorized as {\em primal-only} methods.
They can only find approximately feasible primal solutions that satisfy every constraint up to an {\em additive} error $\eps$, and the running time is proportional to $1/\eps^2$, which poses challenges when a strictly feasible solution is required or when $\eps$ is very small.
The problem is addressed when Arora and Kale~\cite{AKcombinatorial} introduce their {\em primal-dual} SDP framework. 
Their framework is based on the matrix multiplicative weight update method (MMWU) that has been discovered in different fields including optimization~\cite{mmwu:ben2005non}, machine learning~\cite{tsuda}, and online learning~\cite{mmwu:warmuth2006online}. 
It produces a pair of solutions to both primal and dual SDPs, which are strictly feasible and have multiplicative approximation ratios.
Together with specific combinatorial techniques, Arora and Kale also developed efficient approximation algorithms for NP-hard graph problems such as {\sc Sparsest Cut} and {\sc Balanced Separator}. 

The MWU-based frameworks (for both LPs and SDPs) rely on the ``width'' of the problem, a parameter related to the largest entry or eigenvalue of the constraints.
As a closely related area of study, the literature also contains a range of sequential and parallel algorithms that are designed for specific families of LPs and SDPs, known as positive LPs (see, e.g.~\cite{mwu-based:luby1993parallel,mwu-based:young2001sequential,mwu-based:allen2014using}) and positive SDPs~\cite{mwu-based:jain2011parallel, mwu-based:allen2016using}.
Unlike the MWU-based algorithms, the running time of these algorithms is typically independent of the width.
However, these width-independent algorithms often necessitate more sophisticated analysis, and most of the techniques developed in this line of research are not applicable to the generic problems.

\section{Preliminaries}\label{sec:jasc}

\subsection{Symmetric cones and Euclidean Jordan algebras} \label{sec:eja}

A \emph{symmetric cone} is a closed convex cone $\cc{K}$ in a finite-dimensional inner product space $\cj$ that is self-dual (i.e., $\cc{K}^* \triangleq\{\mm{y}\in \cj: \mm{y} \bullet \mm{x} \ge 0, \forall \mm{x}\in \cc{K}\} = \cc{K}
$) and homogeneous (i.e., for any $\mm{u}, \mm{v} \in \mmint(\cc{K})$
there exists an invertible linear transformation $\T:\cj \to \cj$ such that $\T(\mm{u}) = \mm{v}$ and $\T(\cc{K})=\cc{K}$). Importantly, a symmetric cone can also be characterized as the cone of squares of an \emph{Euclidean Jordan algebra} (EJA, e.g.\ see Proposition 2.5.8 in \cite{PHD:V07}){:}
\[ 
    \cc{K} = \{\mm{x}^2 \triangleq \mm{x} \circ \mm{x} : \forall \mm{x} \in \cj\}.
\]
Here, we recall that an EJA is a finite-dimensional vector space $\cj$ equipped with a bilinear product $\circ:\cj\times\cj\to\cj$ that satisfies $\mm{x}\circ \mm{y} = \mm{y}\circ \mm{x}$ and $\mm{x}^2\circ(\mm{x}\circ \mm{y}) = \mm{x}\circ (\mm{x}^2\circ \mm{y})$ for all $\mm{x}, \mm{y} \in \cj$, an inner product $\bullet$ that {satisfies} $(\mm{x} \circ \mm{y}) \bullet \mm{z} = \mm{x} \bullet (\mm{y} \circ \mm{z})$ for all $\mm{x}, \mm{y}, \mm{z}\in\cj$, and an identity element $\mm{e}$ {satisfying} $\mm{e} \circ \mm{x} = \mm{x} \circ \mm{e} = \mm{x}$ for all $\mm{x} \in \cj$. Every EJA has a finite \emph{rank}, which is defined as the largest degree of the minimal polynomial of any of its elements.

Every element $\mm{x}$ of an EJA has a {\em spectral decomposition}, in the sense that there exists unique (up to ordering and multiplicities) real numbers $\lambda_1,\cdots,\lambda_r$ and a Jordan frame $\{\mm{q}_1, \ldots, \mm{q}_r\}$ such that $\mm{x} = \sum_{i=1}^r \lambda_i \mm{q}_i$, where a Jordan frame is a collection of primitive idempotents $\mm{q}_1,\ldots,\mm{q}_r$ that satisfy $
\mm{q}_i\circ \mm{q}_j = \mm{0}\ \forall i\neq j$ and $\sum_{i=1}^r \mm{q}_i = \mm{e},$ and $r$ is the rank of the EJA. An element $\mm{q}\in\cj$ is an idempotent if $\mm{q}^2 = \mm{q}$ and is said to be primitive if it is nonzero and cannot be written as the sum of two nonzero idempotents.
The \emph{trace} of an EJA element with spectral decomposition $\mm{x}=\sum_{i}^r\lambda_i \mm{q}_i $ is given by the sum of its eigenvalues
$\Tr(\mm{x})=\sum_{i}^r\lambda_i$.
In the following, the {\em inner product} of the EJA is denoted by $\mm{x} \bullet \mm{y} = \Tr(\mm{x} \circ \mm{y})$.
The 
{cone of squares}
of an EJA is then the set of elements with non-negative eigenvalues.

The characterization of symmetric cones as the cone of squares of EJAs, taken together with the classification result for finite-dimensional EJAs (e.g.\ see \cite{BOO:FK94}), implies that all symmetric 
{cones over the real field are isomorphic to Cartesian products of 
PSD cones and
second-order cones.}
We now present some examples of symmetric cones, together with their algebraic properties which shall be important in our applications:
\begin{itemize}
    \item {\bf The non-negative orthant $\bR^{d}_{+}$.} The {EJA} product associated with $\bR^{d}_{+}$ is the component-wise product $\mm{x}\circ\mm{y} = \mmdiag(\mm{x})\ \mm{y}$, and the associated identity element is $\mm{e}=\mm{1}_d$. The rank of $\bR^{d}_{+}$ is $d$, and the spectral decomposition of an element $\mmx$ is $\sum_{i=1}^{d} x_i \mm{e}_i$, where $\mm{e}_i$ is the $i$-th standard basis  vector. The trace is then $\Tr(\mm{x}) = \sum_{i=1}^{d} x_i$.
    
    \item {\bf The second-order cone $\cc{Q}^{d+1}$.} We consider $\bR^{d + 1}$ as an  {EJA} with Jordan product  $(\mm{u}; u_0) \circ (\mm{v}; v_0) = \frac{1}{\sqrt{2}} (v_0\mm{u} + u_0\mm{v}; \mm{u}^T\mm{v} + u_0 v_0)$. The corresponding  cone of squares is the second-order cone 
    $\cc{Q}^{d+1} \triangleq \{(\mm{u}; u_0) \in \bR^{d + 1} : \|\mm{u}\|_2 \leq u_0 \}. $
    The rank of this EJA is 2 and its identity element is $\mm{e} = (\mm{0}; \sqrt{2})$.
    For an element $(\mm{u}; u_0)$ of this EJA, its eigenvalues are $\lambda_{1,2} = \frac{1}{\sqrt{2}}(u_0 \pm \|\mm{u}\|_2)$ and the corresponding idempotents are $\mmq_{1,2} = \frac{1}{\sqrt{2}}(\pm \mmv; 1)$, where $\mmv = \frac{\mmu}{\|\mmu\|}$ if $\mmu\neq \mmzero$ and any vector with unit $\ell_2$-norm otherwise. 
    The trace of an element $(\mmu; u_0)$ in this EJA is given by $\Tr(\mm{u}; u_0) = \sqrt{2}u_0$.
    
    \item {\bf The PSD cone $\cc{S}^{d}_+$.} The underlying EJA of the cone of $d\times d$ Hermitian PSD matrices is the vector space $\bR^{d(d+1)/2}$ endowed with the Jordan product $\mmx\circ\mmy = {1\over 2}{\rm vec}(\mmX\mmY + \mmY\mmX)$, where $\mmX = {\rm mat}(\mmx)$ and $\mmY = {\rm mat}(\mmy)$. 
    The rank of the EJA is $d$ and its identity element $\mme = {\rm vec}(I)$, where $I$ is the identity matrix.
    For an element $\mmx$ in the EJA, its spectral decomposition is $\mmx = \sum_{i=1}^d \lambda_i{\rm vec}(\mmv_i\mmv_i^T)$, where $\lambda_i$ and $\mmv_i$ are from the eigendecomposition of ${\rm mat}(\mmx)$. The trace of $\mmx$ is simply $\Tr({\rm mat}(\mmx))$.
    
    \item {\bf Product cones.} The Cartesian products of symmetric cones are also symmetric cones. We denote the Cartesian product of $n$ symmetric cones $\cc{K}_1,\dots, \cc{K}_n$ as $\cc{K} = \cproduct_{i=1}^n \cc{K}_i$. {(We will also use $\cproduct$ to denote vector concatenations.)} The rank/trace of a product cone is defined as the sum of the ranks/traces of its components. The spectral decomposition of the product cone follows from the decomposition of each component.
\end{itemize}

The spectral decomposition allows us to define the \emph{L\"{o}wner extension} {of} 
{the exponentiation}
function $\exp: \R \to \R$, which {gives a} function
that {maps elements of the EJA with spectrum in $\R$ to elements of the EJA as follows:}
 \[
    \mathbf{exp}:\quad \sum_{i=1}^r\lambda_i \mm{q}_i \mapsto  \sum_{i=1}^r \exp(\lambda_i)\mm{q}_i \ \in \cj.
\]
Let $(\cc{J}, \circ)$ be an EJA and $\mmx, \mmy \in \cc{J}$, the generalized Golden-Thompson inequality holds~\cite{Golden_Thompson_EJA}:
\beq\label{eq:golden_thompson}
\Tr(\mmexp(\mmx + \mmy)) \le \Tr(\mmexp(\mmx)\circ \mmexp(\mmy)).
\eeq

Moreover, for any element $\mmx\in\cc{J}$, let $\sum_i\lambda_i\mmq_i$ be its spectral decomposition. We define the infinity norm of $\mmx$ as the maximum magnitude of its eigenvalues, namely $\|\mmx\|_\infty \triangleq \max_{i\in [r]} |\lambda_i|$.

{
Additional details on symmetric cones and EJAs can be found in \cite{BOO:FK94}, \cite{PHD:V07}, and \cite{vanden}.}

\subsection{Symmetric cone programming}\label{sec:scp}
A Symmetric Cone Program (SCP) corresponds to linear optimization over the intersection of a symmetric cone with an affine subspace.  A pair of primal/dual SCPs over a symmetric cone $\cc{K}$ (in standard form) is given by:
\begin{equation}\label{scpbeforesplit}
\begin{aligned}
    \max \quad & \mmc \bullet \mmx\\
    \subto \quad & \mma_j \bullet \mmx = b_j \quad \forall j \in [m]\\
    & \mmx \succeq_\cc{K} 0
\end{aligned}
\hspace{4em}
\begin{aligned}
    \min \quad & \mmb^T \mmy\\
    \subto \quad & \sum_{j=1}^m \mma_j y_j \succeq_\cc{K} \mmc
\end{aligned}
\end{equation}

As a consequence of the classification of EJAs, symmetric cones provide a unified framework to study Linear Programs (LPs), Semidefinite Programs (SDPs), and Second-Order Cone Programs (SOCPs). Specifically, when $\mathcal{K} = \mathbb{R}^d_+$, SCPs correspond to LPs; when $\mathcal{K}$ is a Cartesian product of second-order cones, they correspond to SOCPs; and when $\mathcal{K}$ is the cone of Hermitian PSD matrices $\cc{S}^d_+$, they correspond to SDPs.

\subsection{{Symmetric cone multiplicative weights update (SCMWU)}} \label{sec:scmwu}

SCMWU is a recently introduced algorithm for {online linear optimization (OLO)} over the trace-one slice of a symmetric cone~\cite{scmwu}. It extends the seminal {MWU} method for {OLO} over the simplex~\cite{survey} and the {MMWU} method for OLO over the set of density matrices~\cite{kale,tsuda}.

We now describe   the online linear optimization framework where the SCMWU  finds its application.
At each decision epoch $t$,  a decision maker needs to choose a trace-one element $\xt{\mmp}{t}$ within a symmetric cone $\cc{K}$.   Once this choice is made, a linear loss function
{$\xt{\ell}{t}(\mmp) = \xt{\mmm}{t} \bullet \mmp$, where $\xt{\mmm}{t}$ is in the underlying EJA space of $\cc{K}$, is revealed.}
Employing SCMWU, the next iterate is determined~by
\begin{equation}\label{SCMWU}\tag{\sf SCMWU}
    \xt{\mmp}{t+1} = \frac{\mmexp(-\eta \sum_{\tau=1}^{t}\xt{\mm{m}}{\tau})}{\Tr(\mmexp(-\eta \sum_{\tau=1}^{t}\xt{\mm{m}}{\tau}))},
\end{equation}
and $\xt{\mmp}{1}$ is initialized as $\frac{\mme}{r}$, where $r$ represents the rank of the underlying EJA.%

In the asymptotic limit $t\to +\infty$, the cumulative losses incurred by SCMWU become comparable to the losses {of the best fixed action} that the decision maker {can make} 
{in}
{hindsight.}
This is commonly referred to as the ``no-regret'' property.
Here, we provide the regret bound for SCMWU presented in~\cite{scmwu} for the sake of comprehensiveness.

\begin{theorem}
\label{thm:scmwu}
Let $(\cc{J}, \circ)$ be an EJA of rank $r$, $\cc{K}$ be its cone of squares, and $\mmx\bullet\mmy = \Tr(\mmx \circ \mmy)$ be the EJA inner product. For any $\eta \in (0, 1]$ and any sequence of loss vectors $\xt{\mm{m}}{t}$ satisfy $\|\xt{\mmm}{t}\|_\infty \le 1$, the iterates $\xt{\mm{p}}{t}$ generated from~\eqref{SCMWU} satisfy
\[
\sum_{t=1}^T \xt{\mm{m}}{t}\bullet \xt{\mm{p}}{t} \leq \eigenval{\min}{\sum_{t=1}^T \xt{\mm{m}}{t}} + \eta T + \frac{\ln r}{\eta},
\]
where $\eigenval{\min}{\cdot}$ is the minimum eigenvalue of an EJA vector.
\end{theorem}

\begin{proof}
Let $\xt{\mm{w}}{t} = \mmexp(-\eta \sum_{\tau=1}^{t-1} \xt{\mm{m}}{\tau})$, $\xt{\mm{w}}{1} = \mm{e}$ and $\xt{\mm{p}}{t} = \frac{\xt{\mm{w}}{t}}{\Tr(\xt{\mm{w}}{t})}$. We can establish an upper bound on $\Tr(\xt{\mm{w}}{T+1})$ via the following steps:
\begin{equation}
\label{eq:potential_upper}
\begin{aligned}
   \Tr(\xt{\mm{w}}{T+1}) &= \Tr(\mmexp(-\eta \sum_{t=1}^{T} \xt{\mm{m}}{t}))\\
   &\overset{(a)}{\le} \Tr(\xt{\mm{w}}{T} \circ \mmexp(-\eta \xt{\mm{m}}{T}))\\
   &= \xt{\mm{w}}{T} \bullet \mmexp(-\eta \xt{\mm{m}}{T})\\
   &\overset{(b)}{\le} \xt{\mm{w}}{T} \bullet (\mm{e} - \eta \xt{\mm{m}}{T} + \eta^2 (\xt{\mm{m}}{T})^2 )\\
   &= \Tr(\xt{\mm{w}}{T}) - \eta \xt{\mm{m}}{T} \bullet \xt{\mm{w}}{T} + \eta^2 (\xt{\mm{m}}{T})^2 \bullet \xt{\mm{w}}{T}\\
   &= \Tr(\xt{\mm{w}}{T}) \cdot (1 - \eta \xt{\mm{m}}{T} \bullet \xt{\mm{p}}{T} + \eta^2 (\xt{\mm{m}}{T})^2 \bullet \xt{\mm{p}}{T})\\
   &\overset{(c)}{\le} \Tr(\xt{\mm{w}}{T}) \cdot \exp(- \eta \xt{\mm{m}}{T} \bullet \xt{\mm{p}}{T} + \eta^2 (\xt{\mm{m}}{T})^2 \bullet \xt{\mm{p}}{T})\\
   &\overset{(d)}{\le} r \cdot \exp(- \eta \sum_{t=1}^T \xt{\mm{m}}{t} \bullet \xt{\mm{p}}{t} + \eta^2 \sum_{t=1}^T (\xt{\mm{m}}{t})^2 \bullet \xt{\mm{p}}{t}).
\end{aligned}
\end{equation}
Here $(a)$ is due to the generalized Golden-Thompson inequality (see~\eqref{eq:golden_thompson}), $(b)$ is due to the self-duality of the symmetric cone and the fact that $e^x \le 1 + x + x^2 \ \forall x [-1, 1]$, $(c)$ is because $1 + x \le e^x$, and $(d)$ is by induction.
On the other hand, a lower bound on $\Tr(\xt{\mmw}{T+1})$ can also be established:
\begin{equation}
\label{eq:potential_lower}
\begin{aligned}
   &\Tr(\xt{\mm{w}}{T+1}) = \Tr(\mmexp(-\eta \sum_{t=1}^{T} \xt{\mm{m}}{t})) = \sum_{k=1}^{r} \eigenval{k}{\mmexp(-\eta \sum_{t=1}^{T} \xt{\mm{m}}{t})}\\
   &= \sum_{k=1}^{r} \exp(\eigenval{k}{-\eta\sum_{t=1}^{T} \xt{\mm{m}}{t}}) \ge \exp(\eigenval{\max}{-\eta\sum_{t=1}^{T} \xt{\mm{m}}{t}}) = \exp(-\eta\eigenval{\min}{\sum_{t=1}^{T} \xt{\mm{m}}{t}}),
\end{aligned}
\end{equation}
where $\eigenval{k}{\cdot}$ represents the $k$-th eigenvalue.
Combining (\ref{eq:potential_upper}) and (\ref{eq:potential_lower}) gives
\beq\label{eq:regret_ineq}
\begin{aligned}
\exp(-\eta\eigenval{\min}{\sum_{t=1}^{T} \xt{\mm{m}}{t}}) &\le r \cdot \exp(- \eta \sum_{t=1}^T \xt{\mm{m}}{t} \bullet \xt{\mm{p}}{t} + \eta^2 \sum_{t=1}^T (\xt{\mm{m}}{t})^2 \bullet \xt{\mm{p}}{t})\\
&\le r \cdot \exp(- \eta \sum_{t=1}^T \xt{\mm{m}}{t} \bullet \xt{\mm{p}}{t} + \eta^2 T).
\end{aligned}
\eeq
The last inequality is because $(\xt{\mmm}{t})^2 \bullet \xt{\mmp}{t} \le 1$, since $-\mm{e} \preceq_\cc{K} \xt{\mm{m}}{t} \preceq_\cc{K} \mm{e}$ and $\Tr(\xt{\mmp}{t}) = 1$.
Taking logarithms on both sides of~\eqref{eq:regret_ineq}, scaling by $\frac{1}{\eta}$, and rearranging gives the desired result.
\end{proof}

\subsection{Parallel computation}

The {\em work-span} (or work-depth) model is a framework for conceptualizing and describing parallel algorithms.
In this model, a computation can be viewed as a directed acyclic graph (DAG). The work $W$ of a parallel algorithm is the total number of operations, and the span $S$ is the length of the longest path in the DAG.
Given $P$ processors with shared memory, the computation can be executed in $W/P + O(S)$ time with high probability \cite{tschedule1,tschedule2}. Since $P$ is often asymptotically much smaller than $W$, the running time is often dominated by the total work of the algorithm. On the other hand, the span measures the ability of a parallel algorithm to effectively utilize more processors to improve the performance~\cite{BOO:BDS21}.

In this work, we will use the work-span model to analyze our algorithms under parallel settings.
A primitive problem that frequently arises as a subproblem when designing our parallel algorithms is the information aggregation problem. Given an array $(x_1, \dots, x_n)$ of length $n$ and an associative operator~$\diamond$, we want to compute $x_1 \diamond \dots \diamond x_n$. Let $\cc{I}$ be the set of odd indices in the array. For each $i\in\cc{I}$, we compute $y_{\lceil i/2\rceil} = x_i \diamond x_{i+1}$ (if $n$ is odd, we let $y_{\lceil n / 2 \rceil} = x_n$). Then we get a new array $y$ with length $\lceil n/2 \rceil$ and the problem becomes computing $y_1 \diamond \dots \diamond y_{\lceil n/2 \rceil}$. The problem can be solved recursively using the same approach until the length of the array is reduced to 1. The associated computational DAG of this process form a binary tree, where the total number of operations is $O(n)$ and the longest path in the DAG (the height of the binary tree) is $O(\log n)$.
Therefore, under the setting of the work-span model, the information aggregation problem can be solved in $O(n)$ work and $O(\log n)$ span. 

The information aggregation problem covers many basic operations such as summation and extremum, which are building blocks for our algorithms. Here we elucidate the process of parallelizing the SCMWU method within the work-span model, given its frequent status as the most computationally intensive component in our algorithms. Let $\cc{K}$ be the Cartesian product of several primitive symmetric cones and let $\cc{J}$ be the underlying EJA. The computation of the exponential of an element $\mmx\in \cc{J}$ can be decomposed into problems of computing the exponentials of each of its components. Let $N$ be the number of non-zero entries in $\mmexp(\mmx)$ and let $r$ be the rank of $\cc{J}$. Given the eigenvalues (or the trace) of each component of $\mmexp(\mmx)$, the computation of $\Tr(\mmexp(\mmx))$ is simply a summation problem that is solvable in $O(r)$ work and $O(\log r)$ span, and consequently the normalization of $\mmexp(\mmx)$ can be performed in $O(N)$ work and $O(1)$ span.
By the classification theorem of symmetric cones, any symmetric cone over the real field is isomorphic to the Cartesian product of PSD cones and second-order cones. Thus to analyze the computation of $\mmexp(\mmx)$, it only remains analyzing the exponentials of elements in the spaces of PSD cones and second-order cones.
Let $\cc{Q}^{d+1}$ be a second-order cone and let $(\mmu; u_0)$ be an element in the underlying space $\bR^{d+1}$. The spectral decomposition of $(\mmu; u_0)$ is given by $\lambda_{1,2} = {1\over \sqrt{2}}(u_0 \pm \|\mmu\|_2)$ and $\mmq_{1,2} = {1\over \sqrt{2}} (\pm\mmv; 1)$, where $\mmv = \frac{\mmu}{\|\mmu\|_2}$ or any unit vector if $\mmu = 0$. If $\mmu \neq 0$, To compute the eigenvalues and eigenvectors, the most intensive part is to compute $\|\mmu\|_2$, which is again a summation problem with some simple additional arithmetic and can be computed in $O(d)$ work and $O(\log d)$ span. Once the spectral decomposition is obtained, the exponential can be directly computed in $O(d)$ work and $O(1)$ span. 
In summary, an SCMWU step for $\cc{Q}^{d+1}$ can be performed with $O(d)$ work and $O(\log d)$ span.
In~\cite[Sec.~7]{AKcombinatorial}, Arora and Kale proposed an approximation algorithm for computing the exponential of matrices that lie in the space of the PSD cone. Their algorithm utilizes the power-series expansion of the exponential and the Johnson-Lindenstrauss lemma and only involves matrix multiplications, which can be reduced to a series of inner products and summations and can be directly parallelized within the work-span model. Given a $d\times d$ matrix, Arora and Kale's algorithm can be parallelized in ${\sf poly}(d)$ work and ${\sf polylog}(d)$ span.

\section{Primal-dual meta algorithm for SCPs}\label{sec:algorithm}

In this section, and  throughout this work, we consider a pair of primal/dual SCPs  of the~form
\begin{equation}
    \label{eq:klp}
    \tag{SCP}
    \begin{aligned}
        \max \quad   & \mm{c} \bullet \mm{x} + \mmc' \bullet \mmx'                                   \\
        \subto \quad & \mm{a}_j \bullet \mm{x} + \mma_j' \bullet \mmx' = b_j \quad   \forall j \in [m] \\
                     & \mm{x} \in \cc{K},\ \mmx' \in \cc{K}'
    \end{aligned}
    \hspace{3em}
    \begin{aligned}
        \min \quad   & \mm{b}^T \mm{y}                                           \\
        \subto \quad & \textstyle\sum_{j=1}^{m} \mm{a}_j y_j \succeq_\cc{K} \mmc \\
                    & \textstyle\sum_{j=1}^{m} \mm{a}'_j y_j \succeq_{\cc{K}'} \mmc' 
    \end{aligned}
\end{equation}
where $\cc{K}, \cc{K}'$ are symmetric cones,  $\mmc,\mma_j \in \cc{J}$, $\mmc', \mma_j' \in \cc{J}'$,  and $\cc{J}, \cc{J}'$ are the EJAs 
{associated with}
the cones $\cc{K}, \cc{K}'$ respectively. We assume throughout that~\eqref{eq:klp} exhibit strong duality {(e.g. see~\cite{lcp_duality})}, and use $\OPT$ to denote the shared optimal value of the primal and dual problem. 

Notice that the dual feasible region is defined by two generalized inequality constraints,   
\begin{equation}\label{candf}
\mathcal{F}\ {\triangleq}\ \{\mathbf{y} \in \mathbb{R}^m: \sum_{j=1}^{m} \mathbf{a}_j y_j \succeq_{\mathcal{K}} \mathbf{c}\} \ \text{ and } \ \mathcal{C}\ {\triangleq}\ \{\mathbf{y} \in \mathbb{R}^m: \sum_{j=1}^{m} \mathbf{a}'_j y_j \succeq_{\mathcal{K}'} \mathbf{c}'\}, 
\end{equation} 
where $\mathcal{F}$ is a constraint set that is ``hard to deal with'' while $\mathcal{C}$ is a  constraint set that is ``easy to satisfy''.    In our algorithm, the hard constraint set  $\mathcal{F}$ is  relaxed to a halfspace while for the easy  part  we will need to  to perform linear optimization over the intersection of the set $\mathcal{C}$ and a given~halfspace.

As a concrete  example, consider the case where $\cc{K} = \bR^{n}_+$, $\cc{K}' = \R^m_{+}$, $\mm{a}'_j \in \R^m$ is the $j$-th unit vector, and $\mm{c}' = \mm{0}$. Thehe pair of primal/dual SCPs is then just  a pair of primal/dual LPs:
\[ 
\max \{ \mm{c}^T \mm{x} : \mmA\mmx\le \mmb, \  \mm{x} \ge 0\} \text{  and  } 
\min \{ \mm{b}^T \mm{y} : \mmA^T \mmy - \mmc\ge 0, \ \mmy \ge 0\}.
\]
Here, the dual feasible region is the intersection of a polyhedron $\mathcal{F}=\{ \mmy:     \mmA^T \mmy - \mmc\ge 0\}$ corresponding to the hard set and the non-negative orthant $\mathcal{C}$, which corresponds to the easy set.   Note that  optimizing a linear function over the intersection of the non-negative orthant and a halfspace can be accomplished easily. %

As a second example, we consider the case where $\cc{K} = \cc{S}^{n}_+$, $\cc{K}' = \R^m_{+}$, $\mm{a}'_j \in \R^m$ is the $j$-th unit vector, and $\mm{c}' = \mm{0}$. In this setting, the pair of primal/dual  SCPs defined above  correspond to a pair of primal/dual SDPs:
\[
    \max \{  \mm{c} \bullet \mm{x} : \mm{a}_j \bullet \mm{x} \le b_j\ \forall j\in [m], \ \mm{x} \succeq  0\} \text{ and } 
    \min \{     \mm{b}^T \mm{y}: \textstyle\sum_{j=1}^{m} \mm{a}_j y_j \succeq   \mm{c}, \   \mmy \ge 0\}.
\] 
In this  setting  the hard part of the dual feasible region is carved out by a linear matrix inequality, i.e.\ $\mathcal{F}=\{\mmy: \sum_j\mm{a}_j y_j \succeq   \mm{c}\}$,  while the easy part is again the non-negative orthant. 
 This is exactly  the setting that Arora and Kale consider in their primal-dual framework for SDPs~\cite{AKcombinatorial}.

Finally, our analysis necessitates that the  SCPs adhere to the following assumption:
\[
\exists j\in [m] \text{ such that }  \mm{a}_j=\mm{e}  \text{ and } \mm{a}'_j\in \cc{K}'.
\]
This property allows us to ``project'' an approximately feasible dual solution to a strict one.
For simplicity, we assume that this assumption is true for the last primal constraint, i.e. 
\begin{equation}\label{assumption} 
\mm{a}_m=\mm{e} \  \text{ and } \ \mm{a}'_m\in \cc{K}'.
\end{equation} 
As a canonical example, this assumption holds when we have a bound $R$ on the trace of primal solutions. Indeed,  this case is captured by setting   $\mma_m = \mme$, $\mma'_m \in \cc{K}'$, and $b_m = R$.

In order to solve the~\ref{eq:klp} problem we reduce it to a feasibility problem, i.e.\ we   
{consider} {the following} {problem:} {can} 
{the dual problem in~\eqref{eq:klp} achieve a specified objective value $\alpha$? }
More specifically, does $(1 + \eps)\alpha \ge \OPT$ or $\alpha < \OPT$?
{
In addition to the answer, for the former case we also want to find a dual {feasible} solution to~\eqref{eq:klp} with objective value at most $(1+\eps)\alpha$, and for the latter case a primal {feasible} solution with objective value at least $\alpha$. 
In the following, the decision problem, together with the problem of finding the primal/dual solutions, is referred to as the {\em $\alpha$-feasibility test problem}, $\alpha$-FTP.
Given an algorithm that solves $\alpha$-FTP, the \ref{eq:klp} problem can then be solved by employing a binary search (or other searching methods) to narrow the range of $\alpha$. 
Therefore, in this section, we {discuss mainly} the algorithm for solving~$\alpha$-FTP.
}

 Define  $\cc{S}_\alpha = \{\mmy\in\bR^m : \mmb^T\mmy \le \alpha\}$ to be the $\alpha$-sublevel set of the dual objective function.  Recalling the definitions of  $ \mathcal{C}$  and $\mathcal{F}$ (see \eqref{candf})  we have that 
 \[ 
    \cc{F}\cap\cc{C}\cap \cc{S}_\alpha \neq \varnothing \iff \OPT\le \alpha.
\]
Indeed, if   $\alpha\ge \OPT$,  the dual optimal solution   lies in $  \cc{F}\cap\cc{C}\cap \cc{S}_\alpha$, so it is nonempty.
Conversely, {if} $\alpha < \OPT$,  all   feasible solutions have  value $>\alpha$, so  we have  $\cc{F}\cap\cc{C}\cap \cc{S}_\alpha = \varnothing$.
{Therefore, $\alpha$-FTP} is equivalent to determine whether the intersection between $\cc{F}$ and $\cc{C}\cap \cc{S}_\alpha$ is empty or not. 
Since $\cc{F}$, $\cc{C}$ and  $\cc{S}_\alpha$ are convex sets, $\cc{F}\cap\cc{C}\cap \cc{S}_\alpha = \varnothing$ is equivalent to saying that the hard constraint set $\mathcal{F}$ can be separated from $\cc{C}\cap \cc{S}_\alpha$, which can be certified by a separating hyperplane.  Our meta-algorithm is designed to discover a separating hyperplane, thereby certifying that $\alpha < \OPT$.  To this end, we relax the hard constraint set $\cc{F}$ to a halfspace $\mathcal{H}_+$ and test whether the intersection $\mathcal{H}_+ \cap \mathcal{C} \cap \mathcal{S}_\alpha$ is empty. If this intersection is empty, then $\mathcal{H}$ serves as a certificate that $\alpha < \OPT$. As long as the intersection is non-empty, we adjst the halfspace to ensure it contains the hard constraint set $\mathcal{F}$ while also endeavoring to separate it from the average of all previous iterates.  If we succeed to do this for enough  times,  we will  find  an approximate feasible  solution.
The high-level outline of the algorithm is presented in the boxed environment below.

\bigskip
\begin{nolinenumbers}
{
\begin{center}
\fbox{
    \begin{minipage}{.95\linewidth}
        {\bf Primal-dual meta-algorithm for {$\alpha$-FTP}}

        \medskip
        Let $\xt{\cc{H}}{1}$ be a hyperplane such that $\cc{F}$ is contained  in the   closed positive halfspace,~$\xt{\cc{H}}{1}_+$.

        \medskip
        For each iteration $t$, we test whether the intersection of $\xt{\cc{H}}{t}_+$ with $\cc{C} \cap \cc{S}_\alpha$ is empty:
        \begin{itemize}
            \item If yes, conclude that $\alpha < \OPT${,} {find} a primal feasible solution to~\eqref{eq:klp} with objective value {at least} $\alpha$, {and terminate.}
            
            \item If no, find a point $\xt{\mmy}{t} \in \xt{\cc{H}}{t}_+ \cap \cc{C} \cap \cc{S}_\alpha$. Let $\xt{\bar{\mmy}}{t} = \frac{1}{t}\sum_{\tau=1}^t \xt{\mmy}{\tau}$, which is the mean of the past iterates.
                  Adjust $\xt{\cc{H}}{t}$ to get a new hyperplane $\xt{\cc{H}}{t+1}$ (maintaining $\cc{F}\subseteq \xt{\cc{H}}{t+1}_+$), with the goal of   separating  $\xt{\bar{\mmy}}{t}$ from~$\cc{F}$.
                  
        \end{itemize}    %
        \medskip
        {If the algorithm does not terminate within a sufficiently large number of iterations, output a dual feasible solution to~\eqref{eq:klp} with objective value at most $(1 + \eps)\alpha$. }
    \end{minipage}
}
\end{center}
}
\end{nolinenumbers}
\bigskip

\begin{figure}[tb]
    \centering
    \includegraphics[width=\linewidth]{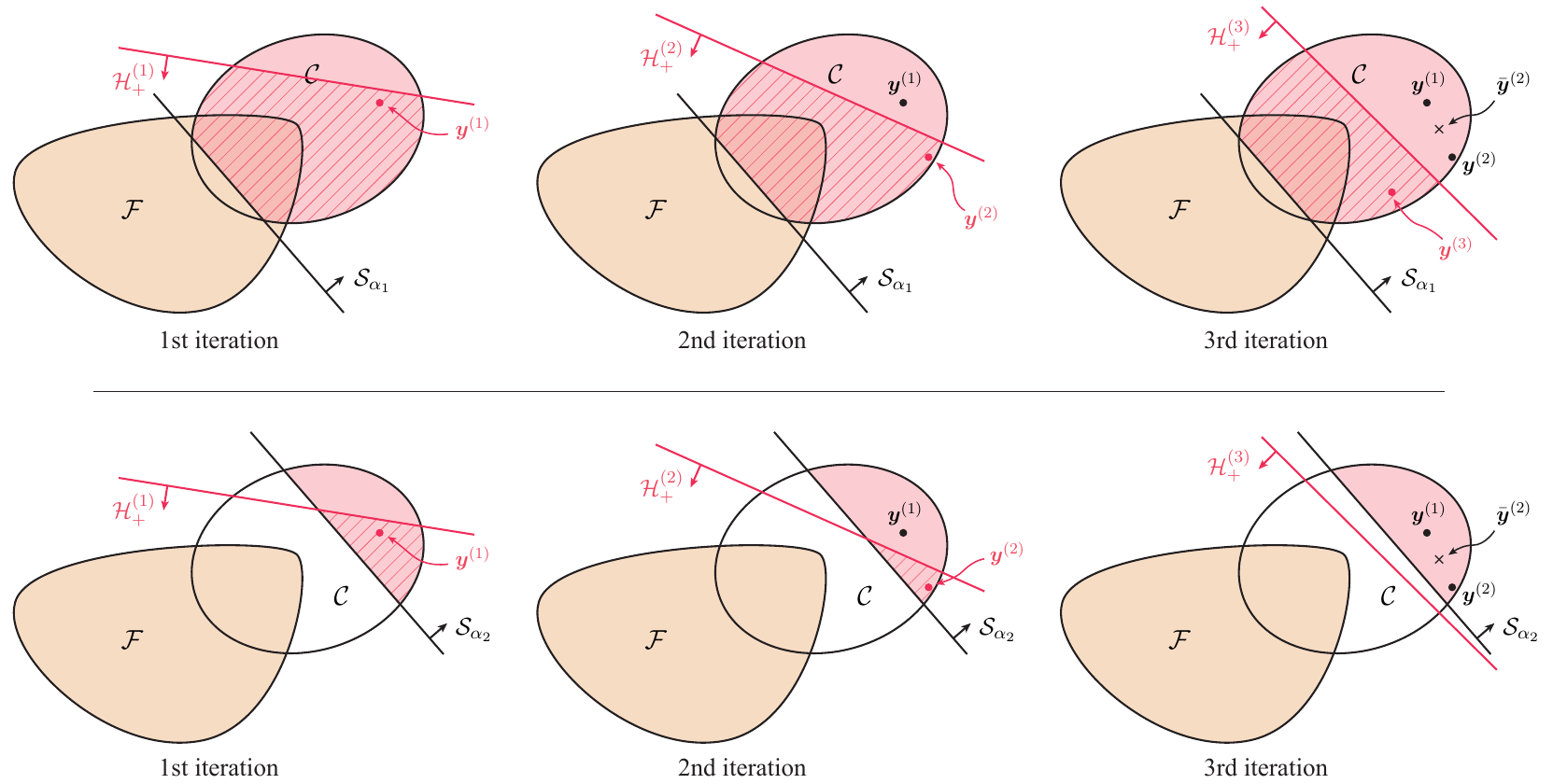}
    \caption{Illustration of the process of the meta-algorithm, where $\alpha_1 > \OPT > \alpha_2$. 
    }
    \label{fig:meta_illustration}
\end{figure}

Figure~\ref{fig:meta_illustration} illustrates three  iterations  (in two settings) 
of  our meta-algorithm  in
both non-separable and separable cases.
{In}
the former case (the top row), where %
{$\cc{F}\cap\cc{C}\cap\cc{S}_{\alpha_1} \not= \varnothing$} 
and $\alpha_1> \OPT$, our algorithm would not find a separating hyperplane{, and would thus} 
continue till $t$ reaches some large number and
{then output a dual feasible solution.}
In the latter case, indicated in the bottom row, the algorithm identifies a separating hyperplane, concluding  that $\alpha_2 < \OPT$. %

We designate this approach as a meta-algorithm because it contains several degrees of freedom that must be specified to develop a concrete algorithm for the \(\alpha\)-feasibility test problem. We will now discuss how each of the steps within the meta-algorithmic framework can be implemented.

\medskip\noindent{\bf Parameterizing the hyperplanes $\xt{\cc{H}}{t}$.}  
For each iteration $t$, we need to find a  halfspace $\xt{\cc{H}}{t}_+$ satisfying $\cc{F} \subseteq \xt{\cc{H}}{t}_+$. 
Recall that $\cc{F}$ is defined by the generalized inequality $\sum_j\mma_j y_j - \mmc \succeq_\cc{K} \mmzero$. A generalized inequality is equivalent to an infinite number of halfspace constraints, i.e.
\[
    \sum_j\mma_j y_j - \mmc \succeq_\cc{K} \mmzero \iff  (\sum_j \mma_j y_j - \mmc) \bullet \mmp \ge 0, \ \forall  \mmp\in  \cc{K} \text{ with }  \Tr(\mmp)=1,
\]
by the self-duality and homogeneous properties of symmetric cones.
Thus, it is a natural choice to consider hyperplanes $\xt{\cc{H}}{t}$ that are parametrized  by  trace-one vectors $\xt{\mmp}{t} \in \cc{K}$. That is, we define
\[
    \xt{\cc{H}}{t}=\{\mmy\in\bR^m : (\textstyle\sum_j \mma_j y_j - \mmc)\bullet \xt{\mmp}{t} = 0\},
\]
and $\cc{F}$ indeed lies in the closed positive halfspace  $\xt{\cc{H}}{t}_+=\{\mmy\in\bR^m : (\sum_j \mma_j y_j - \mmc)\bullet \xt{\mmp}{t} \ge  0\}$. 

\medskip\noindent{\bf Testing of separation.}  Given a halfspace $\xt{\cc{H}}{t}_+$ satisfying $\cc{F} \subseteq \xt{\cc{H}}{t}_+$,   we try to determine whether the intersection {of} $\xt{\cc{H}}{t}_+$ and $\cc{C} \cap \cc{S}_\alpha$ is empty. 
If yes, then $\cc{F}$, which is contained in $\xt{\cc{H}}{t}_+$, also has no intersection with $ \cc{C} \cap \cc{S}_\alpha$ {and we can conclude that $\alpha < \OPT$.}
On the other hand, if  $\xt{\cc{H}}{t}_+\cap \cc{C} \cap \cc{S}_\alpha\ne \varnothing$, we also want to output a {point $\xt{\mmy}{t}$} inside the region as a witness to this fact. 
As this decision problem is the core of the meta-algorithm, we abstract it as an {\sf ORACLE}: 

\begin{nolinenumbers}
{
\begin{center}
\fbox{
\begin{minipage}{.85\linewidth}
    Given $\xt{\mmp}{t}\in \cc{K}$ with $\Tr(\xt{\mmp}{t}) = 1$, determine if there exists $\mmy \in \bR^m$ such that
    \[
        (\textstyle\sum_j \mma_j y_j - \mmc) \bullet \xt{\mmp}{t} \ge 0,\ \mmb^T \mmy \le \alpha \text{ and } \mmy \in \cc{C}.
    \]
\end{minipage}
}
\end{center}
}
\end{nolinenumbers}

The  {\sf ORACLE} can be implemented in a variety of ways.  One possible  implementation of    the  {\sf ORACLE}  is to  solve the following convex problem (with $\xt{\mmp}{t}$ as the input at iteration $t$):
\begin{equation}\label{eq:oracle_program}
    \begin{aligned}
        \max_\mmy \quad & \textstyle\sum_j (\mma_j \bullet \xt{\mmp}{t}) y_j - \mmc \bullet \xt{\mmp}{t} \\
        \subto \quad    & \mmb^T \mmy \le \alpha, \quad \mmy \in \cc{C}.
    \end{aligned}
\end{equation}
If the optimal value is negative, this means we have $(\sum_j \mma_j {y}_j - \mmc) \bullet \xt{\mmp}{t} < 0$ {for all the points $\mmy \in \cc{C}\cap\cc{S}_\alpha$, in which case the entire region of $\cc{C}\cap\cc{S}_\alpha$ is in the open negative halfspace $\xt{\cc{H}}{t}_{-}$ and separated from $\cc{F}$ in $\xt{\cc{H}}{t}_{+}$.}
Thus, the {\sf ORACLE} can output {affirmatively} that 
$\xt{\cc{H}}{t}_{+} \cap \cc{C} \cap \cc{S}_\alpha$ is empty.       
On the other hand, if    the 
optimal value of \eqref{eq:oracle_program} is {non-negative} and  $\xt{\mmy}{t}$ is an optimal solution, we clearly have that $\xt{\mmy}{t} \in \xt{\cc{H}}{t}_{+} \cap \cc{C} \cap \cc{S}_\alpha$.

The complexity of each iteration of our algorithm depends on the complexity of implementing the {\sf ORACLE}, which in turn depends on the  ``complexity'' of the set $\mathcal{C}$. 
Additionally, let $\cc{D} = \{\mmp\in \cc{K} : \Tr(\mmp) = 1\}$ be the set of all possible inputs to the {\sf ORACLE}. We assume that for any input $\mmp\in\cc{D}$, the output $\mmy$ of the {\sf ORACLE} satisfies $\|\sum_j\mma_j y_j - \mmc\|_\infty \le \rho$. The quantity $\rho$ is designated as the {\em width} of the {\sf ORACLE}. The number of iterations of our algorithm will depend on the width of the {\sf ORACLE} as will be revealed in what follows.
These two considerations on the complexity of implementing the {\sf ORACLE} and its width  offer insights into the properties that the constraint set \(\mathcal{C}\) should possess to be considered easy. 

\medskip\noindent{\bf Finding a primal solution.} 
{Consider the pair of primal/dual optimization problems:}
\begin{equation}\label{eq:primal_construction}
    \begin{aligned}
        \min_\mmy \quad & \mmb^T \mmy                                                               \\
        \subto \quad    & \textstyle\sum_j (\mma_j \bullet \xt{\mmp}{t}) y_j \ge \mmc \bullet \xt{\mmp}{t} \\
                        & \mmy \in \cc{C}
    \end{aligned}
    \qquad
    \begin{aligned}
        \max_{z, \mmx'} \quad & (\mmc\bullet\xt{\mmp}{t})z + \mmc'\bullet\mmx'                                       \\
        \subto \quad          & (\mma_j \bullet \xt{\mmp}{t})z + \mma_j' \bullet \mmx' = b_j \quad \forall j \in [m] \\
                              & z\ge 0,\ \mmx'\in \cc{K}'
    \end{aligned}
\end{equation}
denoted  {\sf P} and {\sf D},  respectively.
{Clearly, if the optimal value of~\eqref{eq:oracle_program} is negative (i.e.\ when the {\sf ORACLE} outputs affirmatively), 
the value of {\sf P}  
is greater than $\alpha$. Then, by strong duality, the value of {\sf D} is also greater than $\alpha$.  Let $(\tilde{z}; \tilde{\mmx}')$ be the  optimal solution of~{\sf D}.
Since $\tilde{z}\ge 0$ and $\xt{\mmp}{t}\in \cc{K}$,} 
{$(\tilde{z}\xt{\mmp}{t}; \tilde{\mmx}')$ is} 
{a primal feasible solution to~\eqref{eq:klp} with  value at least $\alpha$.}

\medskip\noindent{\bf Adjusting the hyperplane $\xt{\cc{H}}{t}$.} 
Consider the case when  $\xt{\cc{H}}{t}_+$ intersects $\cc{C}\cap \cc{S}_\alpha$  and the {\sf ORACLE} returns  $\xt{\mmy}{t}\in\cc{C}\cap\cc{S}_\alpha$. Then,  we  need to adjust $\xt{\cc{H}}{t}$ to a new hyperplane  $\xt{\cc{H}}{t+1}$. {This is done by using the symmetric cone multiplicative weights update (SCMWU) method introduced in Section~\ref{sec:scmwu} to update the vector $\xt{\mmp}{t}\in\cc{K}$ that parameterizes $\xt{\cc{H}}{t}$.
Specifically, we define the loss vectors $\xt{\mmm}{\tau} = \frac{1}{\rho}(\sum_j \mma_j \xt{y}{\tau}_j - \mmc)$ for $\tau = 1,\cdots,t$, where $\xt{\mmy}{\tau}$ is the output of the {\sf ORACLE} at the $\tau$-th iteration and $\rho$ is the width of the {\sf ORACLE}. The next vector $\xt{\mmp}{t+1}$ is then obtained from~\eqref{SCMWU}. Note that since $\sum_{\tau}^t \xt{\mmm}{\tau}$ equals $\frac{t}{\rho}(\sum_j\mma_j \xt{\bar{y}}{t}_j - \mmc)$, where $\xt{\bar{\mmy}}{t}$ is the mean of the past points, the next vector $\xt{\mmp}{t+1}$ can be computed by: 
}
\begin{equation}\label{implement}
    \xt{\mmp}{t+1} = \frac{1}{C}\cdot \mmexp\Big({-\eta t \over \rho}(\sum_j \mma_j \xt{\bar{y}}{t}_j - \mmc)\Big),
\end{equation}
where $C$  normalizes $\xt{\mmp}{t+1} $  to have trace one.

We have two notes to make here. First, for any $\mmy\in\cc{F}$, we have $(\sum_j \mma_j \mmy - \mmc) \bullet \xt{\mmp}{t+1} \ge 0$ and thus $\cc{F}$ is still in $\xt{\cc{H}}{t+1}_+$. Second, here is the intuition on  
how the new hyperplane $\xt{\cc{H}}{t+1}$ endeavors to separate $\xt{\bar{\mmy}}{t}$ and $\cc{F}$:
{since the vectors}
$\sum_j \mma_j \xt{\bar{y}}{t}_j - \mmc$ and
$\xt{\mmp}{t+1}$
commute,
let $\lambda_k$ be the eigenvalues of $\sum_j \mma_j \xt{\bar{y}}{t}_j - \mmc$, we have:
\[
    (\sum_j \mma_j  \xt{\bar{y}}{t}_j - \mmc) \bullet \xt{\mmp}{t+1} = \frac{1}{C}\sum_k\lambda_k\exp({-\eta t \lambda_k\over \rho}).
\]
Hence, {on the right-hand side above,} significantly negative eigenvalues are given large positive weights
{in an effort}
to make the above summation negative. From a geometric perspective, the adjustment endeavors to place $\xt{\bar{\mmy}}{t}$ in the open negative halfspace $\xt{\cc{H}}{t+1}_{-}\ \triangleq\ \{\mmy\in\bR^m : (\sum_j \mma_j y_j - \mmc)\bullet \xt{\mmp}{t+1} < 0\}$ so that it is separated from $\cc{F}$ in {$\xt{\cc{H}}{t+1}_+$}.

 Having outlined all the steps in the meta-algorithm, we are now ready to present our central theorem, specifying the required   number of steps  to reach a conclusion regarding whether {\sf OPT} surpasses or falls short of our current estimate $\alpha$.

\begin{theorem}\label{thm:meta}
{Consider a pair of primal/dual~\ref{eq:klp}s over the symmetric cone $\cc{K}\cproduct\cc{K}'$, where the rank of $\cc{K}$ is $r$, and assumption  \eqref{assumption} holds. 
Fix an error parameter $\eps>0$ and a guess $\alpha>0$ for the optimal value {\sf OPT}.
Execute the primal-dual meta-algorithm for $T = \left\lceil\frac{4 R^2\rho^2 \ln r}{\eps^2\alpha^2}\right\rceil$ iterations, adjusting the hyperplane using SCMWU (as defined in \eqref{implement}) with step size $\eta = \sqrt{\frac{\ln r}{T}}$ and  an {\sf ORACLE}  with width $\rho.$} Then,  we either correctly conclude that ${\sf OPT}>\alpha$,  or, if  
the algorithm does not terminate within $T$ iterations we conclude that ${\sf OPT}\le (1+\eps) \alpha$  and  the vector
\[
    \tilde{\mmy} =\left(\xt{\bar{y}}{T}_1, \ldots, \xt{\bar{y}}{T}_{m-1}, \xt{\bar{y}}{T}_m+ \frac{\eps\alpha}{R}\right)
\]
is a dual feasible solution to \ref{eq:klp} with objective value $\mmb^T\tilde{\mmy} \le (1 + \eps)\alpha$.
\end{theorem}

\begin{proof}
    Suppose that the algorithm does not terminate after  $t = 1,\ldots, T$ iterations. By our choice of $\xt{\mmy}{t}$ and $\xt{\mmm}{t} = \frac{1}{\rho} (\sum_j \mma_j \xt{y}{t}_j - \mmc)$, the loss $\xt{\mmm}{t}\bullet \xt{\mmp}{t}$ for SCMWU is always non-negative. By the regret bound of SCMWU (Theorem~\ref{thm:scmwu}), we have:
    \[
        0 \leq \sum_{t=1}^T \frac{1}{\rho} (\sum_j \mma_j \xt{y}{t}_j - \mmc)\bullet\xt{\mm{p}}{t} \leq \eigenval{\min}{\sum_{t=1}^T \frac{1}{\rho} (\sum_j \mma_j \xt{y}{t}_j - \mmc)} + \eta T + \frac{\ln r}{\eta}.
    \]
    Then, scaling both sides of the inequality by $\frac{\rho}{T}$, we have
    \begin{equation}\label{eq:meta_approx_fsb}
        0 \leq \eigenval{\min}{\sum_{j} \mm{a}_j \xt{\bar{y}}{T}_j - \mm{c}} + \eta \rho + \frac{\rho\ln r}{\eta T},
    \end{equation}
    and substituting $T = \left\lceil\frac{4R^2 \rho^2 \ln r}{\eps^2 \alpha^2} \right\rceil$ and $\eta = \sqrt{\frac{\ln r}{T}}$ into the above gives us
    \[
        \eigenval{\min}{\sum_j \mma_j \xt{\bar{y}}{T}_j - \mmc} + \frac{\eps \alpha}{R} \ge 0, \]
        or equivalently
        \begin{equation}\label{lambdamax}
        \sum_{j} \mm{a}_j \xt{\bar{y}}{T}_j \succeq_\cc{K} \mm{c} - \frac{\eps\alpha}{R} \mm{e}.
    \end{equation}
   We now show 
     that the vector $\tilde{\mmy} = (\xt{\bar{y}}{T}_1, \ldots, \xt{\bar{y}}{T}_{m-1}, \xt{\bar{y}}{T}_m+ \frac{\eps\alpha}{R})$ is dual feasible. First, to show that $ \tilde{\mmy}\in \mathcal{F}$, we use  \eqref{lambdamax} 
     and recall that by assumption   \eqref{assumption} we have that  $\mm{a}_m=\mm{e}$. Then, \eqref{lambdamax}~becomes:
    \[
        \sum_{j=1}^{m-1} \mm{a}_j \xt{\bar{y}}{T}_j + ( \xt{\bar{y}}{T}_m+\frac{\eps\alpha}{R}) \mm{e} = \sum_{j=1}^m\mm{a}_j \tilde{y}_j \succeq_\cc{K} \mm{c}. 
    \]
    Next,  we show that  $ \tilde{\mmy}\in \mathcal{C}$. For this, by definition of the {\sf ORACLE} we have that $\xt{{\mmy}}{t} \in \mathcal{C}$,  for all $t=1,\ldots, T$, and by averaging we get that   $\xt{\bar{\mmy}}{T}\in \mathcal{C}$, i.e.
    \begin{equation}\label{finaleq} 
    \sum_j \mma_j' \xt{\bar{y}}{T}_j \succeq_{\cc{K}'} \mmc'.
    \end{equation} 
    Then, by assumption   \eqref{assumption} we have that $ \mma_m'\in \cc{K}'$, and thus \eqref{finaleq} implies that 
    \[
         \sum_{j=1}^m \mma_j' \tilde{y}_j  = \sum_{j=1}^{m-1} \mma_j' \xt{\bar{y}}{T}_j+\mma_m' (\xt{\bar{y}}{T}_m+{\epsilon \alpha\over R}  )\succeq_{\cc{K}'}     \sum_{j=1}^{m-1} \mma_j' \xt{\bar{y}}{T}_j+\mma_m' \xt{\bar{y}}{T}_m  \succeq_{\cc{K}'} \mmc'.
    \]
    Finally, by assumption  \eqref{assumption} we have that  $b_m = R$. Thus,   the objective value of $\tilde{\mmy}$ satisfies:
    \[
        \mmb^T \tilde{\mmy} = \mmb^T \xt{\bar{\mmy}}{T} + R \frac{\eps\alpha}{R} \le \alpha + \eps\alpha = (1 + \eps)\alpha,
    \]
    which completes the proof.
\end{proof}

\medskip\noindent{\bf Remarks on $\cc{C}$.} 
The 
{iteration complexity} 
is quadratic in the width $\rho$ of the {\sf ORACLE} used. Thus, it is desirable to design an {\sf ORACLE} with a small width. As $\cc{C}$ is a constraint of the convex problem solved by the {\sf ORACLE}, $\rho$ in some way depends on the choice of $\cc{C}$. Therefore, to reduce $\rho$, we may shrink the space of $\cc{C}$ by adding extra constraints to the problem. These constraints 
{should}
be redundant to the original problem~\eqref{eq:klp}, but 
strong enough 
{to restrict}
the result of~\eqref{eq:oracle_program} 
so that
{\sf ORACLE} 
{has a} 
smaller $\rho$, as we will see in an example in Section~\ref{sec:pd}. Moreover, if it is not required to find a primal solution ({by solving} (\ref{eq:primal_construction})), the choice of $\cc{C}$ can be more aggressive:
we can {replace}
the set $\cc{C}$ {by} any 
convex set $\cc{C'}\subseteq \cc{C}$ that {contains the dual optimal solution $\mmy^*$,}
since whenever the {\sf ORACLE} {confirms} that $\xt{\cc{H}}{t}_{+} \cap \cc{C}'\cap \cc{S}_\alpha$ is empty we will have the correct outcome that 
$\OPT > \alpha$ (since $y^* \in \cc{F} \cap \cc{C}' \subseteq \xt{\cc{H}}{t}_{+} \cap \cc{C}'$) and 
$\cc{F} \cap \cc{C} \cap \cc{S}_\alpha = \varnothing$.

\medskip\noindent{\bf From feasibility to optimization.}
Assume an upper/lower bound for the value of $\OPT$. (Otherwise, we can employ an exponential search or a constant-approximation algorithm to obtain one.) With our meta-algorithm, a $(1 + \eps)$-approximate solution to~\eqref{eq:klp} can then be computed using a binary search. However, using the desired accuracy $\eps$ as the error tolerance in each binary search step results in an iteration complexity of $O(T\log\frac{1}{\eps})$. Instead of doing that, we can choose the error tolerance $\eps_i$ adaptively for each search step by setting $\eps_i\alpha_i$ to be within a constant fraction of the size of its searching range. The total number of iterations is then dominated by the last step (with the smallest $\eps_i\alpha_i = \Theta(\eps\OPT)$), and is improved to $O(T)$.
Assume we know that ${\sf OPT} > 0$ and lies in the range $[L_0, U_0]$.
In the first step, we run our meta-algorithm, with $\alpha_1 = L_0 + {1\over 3} (U_0 - L_0)$ and $\varepsilon_1={1\over 3\alpha_1}(U_0 - L_0)$.  According to Theorem~\ref{thm:meta}, we either correctly conclude that this  problem is infeasible, in which case ${\sf OPT}> \alpha_1$ and the  search range becomes:
\[
L_0 + {1\over 3}(U_0 - L_0) \le {\sf OPT} \le U_0,
\]
or we conclude that $ {\sf OPT}\le (1+\eps_1) \alpha_1 = L_0 + {2\over 3}(U_0 - L_0)$ (and find a solution with this value), and the search range becomes: 
\[
L_0 \le {\sf OPT} \le L_0 + {2\over 3}(U_0 - L_0).
\]
Setting $L_1 = L_0 + {1\over 3}(U_0 - L_0),\ U_1 = U_0$ if the first case materializes and $L_1 = L_0,\ U_1 = L_0 + {2\over 3}(U_0 - L_0)$ if the second case materializes. Both cases can be  summarized  as follows: 
\[
    L_1 \le {\sf OPT} \le U_1, \ \text{ where } (U_1 - L_1) = {2\over 3}(U_0 - L_0).
\]
Adopting this iterative refinement approach, at the $i$-th step we set $\alpha_i = L_{i-1} + {1\over 3}(U_{i-1} - L_{i-1})$ and $\eps_i = {1\over 3\alpha_i}(U_{i-1} - L_{i-1})$. For the resultant interval $[L_i, U_i]$ we have
\[
    (U_i - L_i) = \Big({2\over 3}\Big)^i (U_0 - L_0).
\]  
Thus, in $O\big(\log (\frac{U_0 - L_0}{\eps\OPT})\big)$ iterations the range size $(U_i - L_i)$ will be less than the desired (additive) error $\eps\OPT$. The very last dual solution will have value at most $(1+\eps){\sf OPT}$, and a scaled version of the $\xt{\mmp}{t}$ that parameterized the last separating hyperplane is a primal solution with value at least $(1-\eps)\OPT$. At each step of the binary search, we invoke Theorem~\ref{thm:meta} with $\eps_i\alpha_i={1\over 3} \left({2\over 3}\right)^{i-1}(U_0 - L_0)$ that decreases geometrically. The total number of iterations needed to solve all $\alpha_i$-feasibility test during the search is then upper bounded by the complexity needed for the very last feasibility test, which is $O\big(\frac{R^2 \rho^2\ln r}{\eps^2\OPT^2}\big)$.

\medskip\noindent{\bf Extension to minimization problems.}
With slight changes, our meta-algorithm also works for minimization problems.
Consider the following variant of~\eqref{eq:klp} where the primal problem is minimizing and the dual problem is maximizing:
\begin{equation}
    \label{eq:min_scp}
    \begin{aligned}
        \min \quad   & \mm{c} \bullet \mm{x} + \mmc' \bullet \mmx'                                   \\
        \subto \quad & \mm{a}_j \bullet \mm{x} + \mma_j' \bullet \mmx' = b_j \quad \forall j \in [m] \\
                     & \mm{x} \in \cc{K},\ \mmx' \in \cc{K}'
    \end{aligned}
    \hspace{3em}
    \begin{aligned}
        \max \quad   & \mm{b}^T \mm{y}                                           \\
        \subto \quad & \textstyle\sum_{j=1}^{m} \mm{a}_j y_j \preceq_\cc{K} \mmc \\
                    & \textstyle\sum_{j=1}^{m} \mm{a}'_j y_j \preceq_{\cc{K}'} \mmc' 
    \end{aligned}
\end{equation}
This time, given a specified objective value $\alpha$, the $\alpha$-feasibility test problem is to determine whether $(1-\eps)\alpha \le \OPT$ or $\alpha > \OPT$. 
For the former case, we provide a dual solution with objective value at least $(1-\eps)\alpha$, and for the latter case we want a primal solution with value at most $\alpha$.
Define the constraint sets
\begin{equation}\label{min_candf}
    \mathcal{F}\ {\triangleq}\ \{\mathbf{y} \in \mathbb{R}^m: \sum_{j=1}^{m} \mathbf{a}_j y_j \preceq_{\mathcal{K}} \mathbf{c}\} \ \text{ and } \ \mathcal{C}\ {\triangleq}\ \{\mathbf{y} \in \mathbb{R}^m: \sum_{j=1}^{m} \mathbf{a}'_j y_j \preceq_{\mathcal{K}'} \mathbf{c}'\},
\end{equation}
and the $\alpha$-superlevel set of the dual objective function $\cc{S}_\alpha = \{\mmy \in \R^m : \mmb^T \mmy \ge \alpha\}$. Similar to the maximization problems, we have
\[
    \cc{F}\cap \cc{C} \cap \cc{S}_\alpha \ne \varnothing \iff \OPT \ge \alpha.
\]
Thus $\alpha$-FTP can be solved by checking the emptiness of the intersection of $\cc{F}, \cc{C}$ and $\cc{S}_\alpha$.
Given a $\xt{\mmp}{t}\in \cc{K}$ with $\Tr(\xt{\mmp}{t}) = 1$, we parameterize the hyperplane $\xt{\cc{H}}{t}$ as
\[
    \xt{\cc{H}}{t} = \{ \mmy \in \R^m : (\mmc - \textstyle\sum_j\mma_j y_j) \bullet \xt{\mmp}{t} = 0\},
\] 
so that the region $\cc{F}$ is kept inside the closed positive halfspace $\xt{\cc{H}}{t}_+$. To test if the hyperplane $\xt{\cc{H}}{t}$ separates $\cc{F}$ from $\cc{C} \cap \cc{S}_\alpha$ (or equivalently whether $\xt{\cc{H}}{t}_+\cap \cc{C}\cap \cc{S}_\alpha = \varnothing$), the {\sf ORACLE} for minimization problems becomes:

\begin{nolinenumbers}
{
\begin{center}
\fbox{
\begin{minipage}{.85\linewidth}
    Given $\xt{\mmp}{t}\in \cc{K}$ with $\Tr(\xt{\mmp}{t}) = 1$, determine if there exists $\mmy \in \bR^m$ such that
    \[
        (\mmc - \textstyle\sum_j \mma_j y_j) \bullet \xt{\mmp}{t} \ge 0,\ \mmb^T \mmy \ge \alpha \text{ and } \mmy \in \cc{C}.
    \]
\end{minipage}
}
\end{center}
}
\end{nolinenumbers}

Similarly, the {\sf ORACLE} can be implemented by solving the following convex problem:
\beq\label{eq:min_oracle}
\begin{aligned}
    \max_\mmy \quad & \mmc \bullet \xt{\mmp}{t} - \textstyle\sum_j (\mma_j \bullet \xt{\mmp}{t}) y_j\\
    \subto \quad & \mmb^T \mmy \ge \alpha,\quad \mmy \in \cc{C}.
\end{aligned}
\eeq
If the optimal value of the above problem is non-negative, and $\xt{\mmy}{t}$ is the solution.
Let $\rho$ be the width of the above {\sf ORACLE} which is defined as an upperbound on $\|\mmc - \sum_j\mma_jy_j\|_\infty$.
We define the loss vector $\xt{\mmm}{t} = {1 \over \rho}(\mmc - \sum_j\mma_j \xt{\mmy}{t})$ and update $\xt{\mmp}{t}$ using SCMWU:
\begin{equation}\label{min_implement}
    \xt{\mmp}{t+1} = \frac{1}{C}\cdot \mmexp\Big({-\eta t \over \rho}(\mmc - \sum_j \mma_j \xt{\bar{y}}{t}_j)\Big),
\end{equation}
where $C$ normalizes $\xt{\mmp}{t+1}$ to be trace-one.
If the value of problem~\eqref{eq:min_oracle} is negative, we know that $\alpha > \OPT$ and a scaled version of $\xt{\mmp}{t}$ serves as a primal solution of~\eqref{eq:min_scp} with objective value at most $\alpha$.
The running time of the meta-algorithm for solving $\alpha$-FTP for the minimization problem is concluded as the following theorem.

\begin{theorem}\label{thm:min_meta}
{Consider problem~\eqref{eq:min_scp} over the symmetric cone $\cc{K}\cproduct\cc{K}'$, where the rank of $\cc{K}$ is $r$, and assumption~\eqref{assumption} holds. 
Fix an error parameter $\eps>0$ and a guess $\alpha>0$ for the optimal value {\sf OPT}.
Execute the primal-dual meta-algorithm for $T = \left\lceil\frac{4 R^2\rho^2 \ln r}{\eps^2\alpha^2}\right\rceil$ iterations, adjusting the hyperplane using SCMWU (as defined in \eqref{min_implement}) with step size $\eta = \sqrt{\frac{\ln r}{T}}$ and an {\sf ORACLE}  with width $\rho$}. Then, we either correctly conclude that ${\sf OPT}<\alpha$,  or, if  
the algorithm does not terminate within $T$ iterations, we conclude that ${\sf OPT}\ge (1-\eps) \alpha$  and  the vector
\[
    \tilde{\mmy} =\left(\xt{\bar{y}}{T}_1, \ldots, \xt{\bar{y}}{T}_{m-1}, \xt{\bar{y}}{T}_m - \frac{\eps\alpha}{R}\right)
\]
is a dual feasible solution of~\eqref{eq:min_scp} with objective value $\mmb^T\tilde{\mmy} \ge (1 - \eps)\alpha$.
\end{theorem}

\section{Applications}\label{sec:application}
We apply the meta-algorithm to derive novel algorithms 
{for}
two geometry applications: the smallest enclosing sphere problem (Section~\ref{sec:ses}) and the support vector machine problem (a.k.a.\ polytope distance, Section~\ref{sec:pd}). 
{By {\em meta-algorithm}, we refer to the meta-algorithm for $\alpha$-FTP that incorporates the necessary search of a good $\alpha$ to solve~\eqref{eq:klp}} or~\eqref{eq:min_scp}.
For each application, we also show the analysis of the algorithm in the parallel setting.

\subsection{Smallest enclosing sphere}\label{sec:ses}
In  the  Smallest Enclosing Sphere (SES) problem,  the goal is to find the sphere with the minimum radius that encloses all the spheres in the input set
 $\cc{I} = \{S(\mmv_1, \gamma_1), \dots, S(\mmv_n, \gamma_n)\}$, where $S(\mmv_i, \gamma_i)$ is a sphere centered at $\mmv_i$ with radius $\gamma_i$. 
SES is an important problem arising in many branches of computer science such as computer graphics, machine learning, and operations research. In the past few decades, numerous methods have been proposed for solving the SES problem.
These methods can be grouped under two classes: 
{exact and approximation methods.}
The exact methods (for example,~\cite{ses_exact:megiddo1983linear, ses_exact:megiddo1984linear,ses_exact:dyer1992class, ses_exact:welzl2005smallest, ses_exact:gartner1995subexponential, ses_exact:gartner2000efficient}) mainly focus on solving the SES problem (for a set of points) in some fixed (low) dimension where the running time is usually linear in the input size $n$. However, when extending these exact methods to higher dimensions, the dependence on dimension could be sub-exponential or worse. On the other hand,  
the approximation methods are considerably faster in general (high) dimensions, and can be readily adapted to diverse input types such as spheres and ellipsoids.
A notable set of approximation methods uses core-sets~\cite{ses_coreset:badoiu2003smaller,ses_socp:kumar2003comuting, ses_coreset:clarkson2010coresets,ses_coreset:yildirim2008two,ses_coreset:nielsen2009approximating}, to compute an $(1 + \eps)$-approximate solution using only a subset of the input points of size $O(\frac{1}{\eps})$ or $O(\frac{1}{\eps^2})$, and the running time is usually linear in the input size, i.e. $O(\frac{nd}{\eps})$ or $O(\frac{nd}{\eps^2})$, where $d$ is the dimensionality of the problem. 
Note that
{the}
SES problem can also be formulated as a SOCP. However, employing the interior point methods on SOCPs often requires solving a linear system in each iteration and results in a super-quadratic or quadratic dependence in the dimensionality: see, for example,~\cite{ses_socp:kumar2003comuting,ses_socp:zhou2005efficient}.

We now formulate the SES problem as an SCP and use our meta-algorithm to develop an algorithm with a nearly linear running time.  Clearly, SES  can be formulated  as the following optimization problem: %
\begin{equation}\label{eq:ses}\tag{SES}
\begin{aligned}
    \min_{\mmu, \gamma} \quad & \gamma\\
    \subto \quad & \|\mmu - \mmv_i\|_2 \le \gamma - \gamma_i \quad \forall i \in [n]
\end{aligned}
\end{equation}
Since the constraint $\|\mmu - \mmv_i\|_2 \le \gamma - \gamma_i$ can be interpreted as a generalized inequality constraint w.r.t.\ $\cc{Q}^{d+1}$,
the problem can be viewed as a SCP where the cone is the Cartesian product of $n$ second-order cones, $\cproduct_{i=1}^n \cc{Q}^{d+1}$. 
Let $(\mmu^*, \gamma^*)$ be the optimal solution to the above program, then the SES of $\cc{I}$ is centered at $\mmu^*$ and the optimal radius is $\gamma^*$. 

In the algorithm, we use the constraint $\|\mmu-\mmv_1\|_2 \le \gamma - \gamma_1$ to define the set $\cc{C}$ and the remaining constraints to define the set $\cc{F}$. 
Notice that the vector $\xt{\mmp}{t}$ that parameterizes the hyperplane $\xt{\cc{H}}{t}$ lies in the product cone $\cproduct_{i=2}^{n} \cc{Q}^{d-1}$.
The running time of our algorithm is stated below.

\begin{theorem}\label{thm:ses}
Applying the meta-algorithm to problem~\eqref{eq:ses} for $n$ spheres in a $d$-dimensional space, we can compute an enclosing sphere with radius at most $(1 + \eps)\gamma^*$ in $O\left(\frac{nd \log n}{\eps^2}\right)$ time.
\end{theorem}

\begin{proof}
Noting that 
\[
  \|\mmu-\mmv_i\|_2 \le \gamma - \gamma_i \iff (\mmu;\gamma)-(\mmv_i;\gamma_i)\in \cc{Q}^{d+1},
\]
we can write the \eqref{eq:ses} problem as an SOCP:
\[
  {\min_{\mmu, \gamma}} \left\{  \gamma:    (\mmu; \gamma) \succeq_{\cc{Q}^{d+1}} (\mmv_i; \gamma_i), \ i=1,\ldots, n\right\}.
\]
For the \eqref{eq:ses} problem, distinguishing between easy and hard constraints is very natural. The ``easy''  portion is derived by keeping only one of the second-order cone constraints (for instance, the first one), while the ``hard''  part is defined by  the intersection of all other constraints, i.e.
  \[\begin{aligned}
    \cc{C} &= \{(\mmu;\gamma)\in \bR^{d+1}: \|\mmu - \mmv_1\|_2 \le \gamma - \gamma_1 \},\\
    \text{and} \quad
    \cc{F} &= \{(\mmu; \gamma) \in \bR^{d+1} : \|\mmu - \mmv_i\|_2 \le \gamma - \gamma_i, \ \text{for}\ i = 2,\dots, n \}.
\end{aligned}\]

Following this splitting, we express the~\eqref{eq:ses} problem as a pair of primal/dual SCPs in the form of~\eqref{eq:klp}.  Setting   $\cc{K} = \cproduct_{i=2}^{n} \cc{Q}^{d+1}$ and $\cc{K}' = \cc{Q}^{d+1}$, the dual formulation is given by:
\begin{equation}\label{dualses}
\begin{aligned} 
        \min \quad   & {\underbrace{({\bf 0};1)}_{\mmb}}^T \underbrace{(\mm{u};\gamma)}_{\mmy}                                           \\
        \subto \quad & \Big( \sum_{j=1}^{d} u_j \underbrace{ \cproduct_{i=2}^n (\mm{e}_j;0)}_{\mma_j} \Big) + \gamma \underbrace{\cproduct_{i=2}^n({\bf 0}; 1)}_{\mma_{d+1}} \succeq_\cc{K}  \underbrace{ \cproduct_{i=2}^n(\mmv_i;\gamma_i)}_{\mmc} \\
& \Big( \sum_{j=1}^{d} u_j \underbrace{  (\mm{e}_j;0)}_{\mma'_j} \Big) + \gamma \underbrace{ ({\bf 0}; 1)}_{\mma'_{d+1}} \succeq_{\cc{K}'}  \underbrace{ (\mmv_1;\gamma_1)}_{\mmc'}. 
                   \end{aligned}
                    \end{equation}
 On the other hand, the primal formulation is given by:
                
\begin{equation}\label{primalses}
\begin{aligned}
\max \quad   & \sum_{i=1}^n (\mmv_i; \gamma_i)\bullet \mmx_i
\\  \subto \quad 
& \underbrace{ \Big( \cproduct_{i=2}^n (\mm{e}_j;0) \Big)}_{\mma_j} \bullet \mm{x}+ \underbrace{ (\mm{e}_j;0)}_{\mma'_j}\bullet \mm{x}'=0, \ j=1,\ldots,d\\ 
& 
\underbrace{ \Big(\cproduct_{i=2}^n ({\bf 0};1)\Big)}_{\mma_{d+1}} \bullet\mm{x}+\underbrace{ ({\bf 0}; 1)}_{\mma'_{d+1}} \bullet \mm{x}'=1\\   
& \mm{x}=(\mm{x}_1; \ldots; \mm{x}_{n-1}) \in \cc{K}, \ \mmx' = \mmx_n \in \cc{K}'
\end{aligned}
\end{equation} 

Examining \eqref{primalses} and \eqref{dualses} we see that assumption \eqref{assumption} is implicitly satisfied. Specifically, if we scale the last constraint in~\eqref{primalses} by $\sqrt{2}$, we get an equivalent problem where the constraint becomes
\[
    \underbrace{ \Big(\cproduct_{i=2}^n ({\bf 0}; \sqrt{2})\Big)}_{\mma_{d+1}} \bullet\mm{x}+\underbrace{ ({\bf 0}; \sqrt{2})}_{\mma'_{d+1}} \bullet \mm{x}' = \sqrt{2},
\]
and we have $m=d+1$, $\mm{a}_{d+1} = \cproduct_{i=2}^n ({\bf 0}; \sqrt{2}) = \mm{e}$ (the identity in $\cc{K}$) and $\mm{a}'_{d+1}=({\bf 0};\sqrt{2})\in \cc{K}'$. Moreover, we see $b_{d+1} = R = \sqrt{2}$. Note that we don't need to perform the scaling explicitly, but the scaled constraint helps in deriving the final result.

Defining $D = \max_{i\ge 2} \|\mmv_1 - \mmv_i\|_2 + \gamma_1 + \gamma_i$, then a sphere centered at {$\mmv_1$}
with radius $D$ will enclose $\cc{I}$, so $\gamma^* \le D$. The distance of the farthest pair of points in $\cc{I}$ is at least $D$, so $\gamma^* \ge \frac{D}{2}$. To solve the SES problem, we can search for the optimal radius in the range $[\frac{D}{2}, D]$.

Fix an error parameter $\eps>0$ and a guess $\alpha>0$ for the optimal value {\sf OPT}.   Let $\xt{\mmp}{t}$ be the vector that parameterizes $\xt{\cc{H}}{t}$ in the $t$-th iteration: it lies in the product cone $\cc{K} = \cproduct_{i=2}^n \cc{Q}^{d+1}$ that corresponds to $\cc{F}$. Let $(\xt{\mmw}{t}_i; \xt{s}{t}_i) \in \cc{Q}^{d+1}$, where $\xt{\mmw}{t}_i \in \bR^d$ and $\xt{s}{t}_i\in \bR$, be the $(i-1)^{th}$ component of $\xt{\mmp}{t}$. 
In other words, $\xt{\mmp}{t} = \cproduct_{i=2}^n (\xt{\mmw}{t}_i; \xt{s}{t}_i)$.
The convex program that  the {\sf ORACLE} needs to solve to check whether $\xt{\cc{H}}{t}_+ \cap \cc{C}\cap \cc{S}_\alpha = \varnothing$ (recall \eqref{eq:oracle_program})
is given by:
    \[\begin{aligned}
        \max_{\mmu, \gamma} \quad & \sum_{i=2}^n (\mmu; \gamma) \bullet (\xt{\mmw}{t}_i; \xt{s}{t}_i) - \sum_{i=2}^n (\mmv_i; \gamma_i) \bullet (\xt{\mmw}{t}_i; \xt{s}{t}_i) \\
        \subto \quad & \gamma \le \alpha, \quad \|\mmu - \mmv_1\|_2 \le \gamma - \gamma_1.
    \end{aligned}\]
    Since $\sum_{i=2}^n \xt{s}{t}_i = \frac{1}{\sqrt{2}}\Tr(\xt{\mmp}{t}) = \frac{1}{\sqrt{2}}$ (see Section~\ref{sec:eja}), the problem can be simplified to:
    \begin{equation}\label{eq:ses_oracle}
        \begin{aligned}
            \max_{\mmu, \gamma} \quad & \left(\sum_{i=2}^n \xt{\mmw}{t}_i\right)^T \mmu + \frac{\gamma}{\sqrt{2}} - \sum_{i=2}^n \mmv_i^T \xt{\mmw}{t}_i - \sum_{i=2}^n \gamma_i \xt{s}{t}_i \\
            \subto \quad              & \gamma \le \alpha, \quad \|\mmu - \mmv_1\|_2 \le \gamma - \gamma_1
        \end{aligned}
    \end{equation}
Notice that once $(\xt{\mmw}{t}_i; \xt{s}{t}_i)$ are fixed, the last two terms in the objective are constants, and increasing the value of $\gamma$ only relaxes the constraint for
{$\mmu$, resulting in}
a larger objective value. Thus, the optimal value of~\eqref{eq:ses_oracle} will be achieved at $\gamma = \alpha$. Once $\gamma$ is fixed, the problem boils down to optimizing a linear function over a ball, which admits a simple closed-form solution.  

Next, we show  that the width $\rho$ of the {\sf ORACLE} is at most $\frac{3D}{\sqrt{2}}$. As explained above, at each iteration the {\sf ORACLE} returns a point $(\mmu;\alpha)$ that satisfies
\begin{equation}\label{oracleineq}  
  \|\mmu - \mmv_1\|_2 \le \alpha - \gamma_1.
\end{equation}
The width then equals the maximum value of 
\[                  %
  \Big|\lambda_k \big(\cproduct_{i=2}^n ({\mmu} - \mmv_i; {\alpha}  - \gamma_i )\big)\Big|,
\]
over all $\mmu, k$ that satisfy
$\|\mmu - \mmv_1\|_2 \le \alpha - \gamma_1$ and $k \in [r]$, which is
\[
  \max_{i\ge 2} \frac{1}{\sqrt{2}}(\alpha - \gamma_i + \| {\mmu}  - \mmv_i\|_2).
\]
Finally, we have  
\[\begin{aligned}
  &\max_{i\ge 2} \frac{1}{\sqrt{2}} (\alpha- \gamma_i + \|\mmu - \mmv_i\|_2) 
  \overset{(a)}{\le} \frac{1}{\sqrt{2}}(\alpha + \|\mmu - \mmv_1\|_2 + \max_{i\ge 2}\|\mmv_1 - \mmv_i\|_2)\\
  &\overset{(b)}{\le} \frac{1}{\sqrt{2}} (\alpha + \alpha - \gamma_1 + \max_{i\ge 2}\|\mmv_1 - \mmv_i\|_2)
  \overset{(c)}{\le} \frac{1}{\sqrt{2}}(2 \alpha + D)  \overset{(d)}{\le} \frac{3D}{\sqrt{2}},
\end{aligned}\]
where $(a)$ is due to the triangle inequality, $(b)$ follows from~\eqref{oracleineq}, $(c)$ follows from the definition of $D$, and $(d)$ is because $\alpha \le D$.

By Theorem \ref{thm:meta} we know that if we execute  the primal-dual meta-algorithm for $T = \left\lceil\frac{4 R^2\rho^2 \ln r}{\eps^2\alpha^2}\right\rceil$ iterations, we either correctly conclude infeasibility or we find a solution with value $(1+\eps)\alpha$. For the latter case, a sphere centered at $\bar{\mmu} = \frac{1}{T}\sum_t \xt{\mmu}{t}$ with radius $(1+\eps)\alpha$ will enclose all spheres in $\cc{I}$.
As discussed above, for problem~\eqref{eq:ses}, we have $R = \sqrt{2}, \rho\le \frac{3D}{\sqrt{2}}$ and $r = 2(n-1)$. The number of iterations $T$ can be simplified to $\left\lceil\frac{36D^2\ln (2n - 2)}{\eps^2\alpha^2}\right\rceil$, which is $O\left(\frac{D^2\log n}{\eps^2\alpha^2}\right)$.

We now analyze the overall running time of the algorithm. For the SES problem, since we obtained a natural lower and upper bound on the value of $\alpha$, where $L_0 = \frac{D}{2}$ and $U_0 = D$. Therefore, the technique of reducing optimization problems to feasibility problems discussed in Section~\ref{sec:algorithm} applies. The total number of iterations required for solving the optimization problem is dominated by the number of iterations required for the last feasibility problem, which is $O(\frac{D^2\log n}{\eps^2\OPT^2})$ and can be further simplified to $O(\frac{\log n}{\eps^2})$ because $\OPT \ge \frac{D}{2}$. The computation in each iteration of the algorithm includes: $(a)$ Solving~\eqref{eq:ses_oracle} for test of separation. This involves computing the coefficients of the objective function and testing the intersection of a halfspace and a hyperball, which requires $O(nd)$ time. $(b)$ Performing a {SCWMU} step for adjusting the hyperplane. This involves computing the loss vector and its spectral decomposition and performing the exponentiation and normalization step, which requires $O(nd)$ time in total. 
Combining the iteration complexity and the running time of each iteration, we conclude that the overall running time of our algorithm is $O(\frac{nd \log n}{\eps^2})$.
\end{proof}

\begin{theorem}\label{thm:ses_parallel}
The algorithm described in Theorem~\ref{thm:ses} can be parallelized under the setting of the work-span model.
The parallel algorithm requires $O\Big(\frac{nd \log n}{\eps^2}\Big)$ work and has $O\Big(\frac{\log(nd)\log n}{\eps^2}\Big)$ span.
\end{theorem}

\begin{proof}
To parallelize the algorithm for the SES problem, we parallelize the computation during each iteration. The computation in each iteration of the algorithm involves computing the analytical solution of~\eqref{eq:ses_oracle} and updating an element of the product cone $\cproduct_{i=2}^n \cc{Q}^{d+1}$ via {SCMWU}, which only requires simple arithmetic and information aggregation operations (summations or extremum queries) on vectors, and can be trivially parallelized.
The work of each iteration of the parallel algorithm is the same as the time complexity of the sequential version, while the span (or parallel depth) becomes $O\left(\log(nd)\right)$.
\end{proof}

\subsection{Support vector machine}\label{sec:pd}
The Support Vector Machine (SVM) is one of the most successful machine learning tools.
{Its}
efficacy is rooted in its inherent simplicity, coupled with exceptional generalization properties applicable to both classification and regression tasks.
The vanilla SVM problem~\cite{svm:cortes1995support} is to find a hyperplane that separates two sets of data points with maximum margin, which is closely related to the geometric problem of finding the closest points in two convex polytopes (a.k.a. polytope distance)~\cite{svm_goemetry:bennett2000duality}. The soft-margin variants of SVM, such as $C$-SVM~\cite{svm:cortes1995support} and $\nu$-SVM~\cite{v-svm:scholkopf2000new}, can also be interpreted as polytope distance (PD) problems for reduced polytopes~\cite{svm_goemetry:bennett2000duality,svm_goemetry:bi2003geometric, svm_geometry:NIPS1999_7fea637f}.
The algorithms designed for SVM are usually based on optimization techniques. A milestone in this direction is the proposal of the sequential minimal optimization (SMO) method~\cite{svm_smo:platt1998sequential}, which can be viewed as a block coordinate descent method and is widely used in practice. Another line of research focuses on solving the PD problem using geometric 
{approaches: remarkable}
algorithms in this modality include the Gilbert-Schlesinger-Kozinec (GSK) method~\cite{svm_gsk:franc2003iterative} and the Mitchell-Dem'yanov-Malozemov (MDM) method~\cite{svm_mdm:mitchell1974finding}. In~\cite{svm_gsk:gartner2009coresets}, the authors also proved that the size of core-set for the SVM/PD problem is $\Theta(\frac{E}{\eps})$, where $E$ is the excentricity of the problem
{instance, which is a quantity that measures its difficulty.}

In this section we study the SVM problem in the separable case. We 
{formulate the problem}
as a pair of primal/dual SCPs and apply the meta-algorithm to develop an algorithm that produces solutions for both problems in nearly linear time.
Given two point sets $\cc{P} = \{\mmu_1, \dots, \mmu_{n_1}\}$ and $\cc{Q} = \{\mmv_1, \dots, \mmv_{n_2}\}$ in $\bR^{d}$, 
we define the matrices $\mmP = (\mm{u}_1, \dots, \mm{u}_{n_1})$ and  $\mmQ = (\mm{v}_1, \dots, \mm{v}_{n_2})$,
where $\mmu_i$ is the $i$-th column of $\mmP$ and, similarly, $\mmv_i$ the $i$-th column of $\mmQ$. We can write 
SVM as the optimization problem {(e.g.\ see~\cite{svm_socp:shivaswamy2006second})}:
\begin{equation}
    \tag{SVM}
    \label{eq:svm}
    \begin{aligned}
        \max_{\mmw, s_1, s_2}\quad & s_1 + s_2                             \\
        \subto\quad                & \mmP^T \mmw - s_1 \mmone \ge \mmzero  \\
                                   & -\mmQ^T \mmw - s_2 \mmone \ge \mmzero \\
                                   & \|\mmw\|_2 \le 1
    \end{aligned}
\end{equation}
{If $(\mmw^*; s_1^*; s_2^*)$ is}
the optimal solution to the above problem, 
{then}
the maximum margin of SVM is given by $s_1^* + s_2^*$ while $\mmw^*$ is the normal vector of the separating hyperplane.
The counterpart of~\eqref{eq:svm} is the PD problem, whose objective is to compute the minimum distance between the convex polytopes corresponding to $\cc{P}$ and $\cc{Q}$:
\begin{equation}
    \tag{PD}
    \label{eq:pd}
    \begin{aligned}
        \min \quad   & \|\mmP \mmmu - \mmQ \mmgamma\|_2       \\
        \subto \quad & \mmone^T \mmmu = \mmone^T \mmgamma = 1 \\
                     & \mmmu, \mmgamma \geq \mmzero
    \end{aligned}
\end{equation}
{In the following, we view~\eqref{eq:svm} as a dual SCP in the form of~\eqref{eq:min_scp}.
The following lemma is useful for us to add additional easy constraints to the problem that help to design the {\sf ORACLE} without altering the solution to the original problem. 

\begin{lemma}\label{lem:svm_bound}
    Let $D$ be the maximum norm of the input points in $\cc{P}$ and $\cc{Q}$. 
    If $(\mmw; s_1; s_2)$ is a feasible solution of~\eqref{eq:svm}, then $s_1, s_2 \le D$.
\end{lemma}
\begin{proof}
    Let $(\mmw; s_1; s_2)$ be any feasible solution to (\ref{eq:svm}).
    From the constraint $\mmP^T\mmw - s_1\mmone \ge 0$, we see that for any $\mmu_i \in \cc{P}$:
    \[
        s_1 \le \mmu_i^T \mmw \le \|\mmu_i\|_2.
    \]
    Since $\|\mmu_i\|_2 \le D \ \forall \mmu_i \in \cc{P}$, this gives:
    \[
        s_1
        \le \max_{\mmu_i\in \cc{P}} \|\mmu_i\|_2 \le D.
    \]
    In the same way, one can prove that $s_2\le D$.
\end{proof}

\begin{theorem}\label{thm:svm}
Let $E = \frac{D^2}{\OPT^2}$ be an instance-specified parameter.
Applying the primal-dual meta-algorithm to problem~\eqref{eq:svm}, we can compute a $(1-\eps)$-approximate solution to~\eqref{eq:svm} and a $(1 + \eps)$-approximate solution to~\eqref{eq:pd} in $O\Big(\frac{E\, n d\log n}{\eps^2}\Big)$ time, where $n = n_1 + n_2$.
\end{theorem}

\begin{proof}
    By Lemma~\ref{lem:svm_bound}, problem~\eqref{eq:svm} is equivalent to:
    \beq\label{eq:svm_new_form}
    \begin{aligned}
        \max_{\mmw, s_1, s_2}\quad & s_1 + s_2                             \\
        \subto\quad                & \mmP^T \mmw - s_1 \mmone \ge \mmzero  \\
                                   & -\mmQ^T \mmw - s_2 \mmone \ge \mmzero \\
                                   & \|\mmw\|_2 \le 1,\ s_1 \le D,\ s_2 \le D.
    \end{aligned}
    \eeq
    It is natural to split the constraints into ``easy'' and ``hard'' parts as follows:
    \beq
    \begin{aligned}
        \cc{C} &= \{(\mmw; s_1; s_2) \in \bR^{d + 2} : \|\mmw\|_2\le 1, s_1 \le D, s_2\le D \}, \\
        \text{and}\quad    
        \cc{F} &= \{(\mmw; s_1; s_2) \in \bR^{d + 2} : \mmP^T \mmw - s_1 \mmone \ge \mmzero,\ -\mmQ^T \mmw - s_2 \mmone \ge \mmzero\}.
    \end{aligned}
    \eeq
Following this splitting, we express problem~\eqref{eq:svm_new_form} as a dual SCP in the form of~\eqref{eq:min_scp}.
Let $\mmP_j$ be the $j$-th column of $\mmP$ and $\mmQ_j$ the $j$-th column of $\mmQ$.
Setting $\cc{K} = \R^{n}_+$
and $\cc{K}' = \cc{Q}^{d+1}\cproduct \R^2_+$, the formulation is given by:
\beq\label{eq:svm_dual_scp}
\begin{aligned}
    \max \quad & {\underbrace{(\mmzero; 1; 1)}_{\mmb}}^T \underbrace{(\mmw; s_1; s_2)}_{\mmy}\\
    \subto \quad & \Big(\sum_{j=1}^d w_j \underbrace{(-\mmP^T_j; \mmQ^T_j)}_{\mma_j}\Big) + s_1 \underbrace{(\mmone_{n_1}; \mmzero)}_{\mma_{d+1}} + s_2 \underbrace{(\mmzero; \mmone_{n_2})}_{\mma_{d+2}} \preceq_\cc{K} \underbrace{(\mmzero; \mmzero)}_{\mmc} \\
    & \Big(\sum_{j=1}^d w_j \underbrace{(\mme_j; 0)\cproduct(0; 0)}_{\mma_j'} \Big) + s_1 \underbrace{(\mmzero; 0)\cproduct (1; 0)}_{\mma_{d+1}'} + s_2 \underbrace{(\mmzero; 0) \cproduct (0; 1)}_{\mma_{d+2}'} \preceq_{\cc{K}'} \underbrace{(\mmzero; 1)\cproduct (D; D)}_{\mmc'}
\end{aligned}
\eeq
On the other hand, the primal formulation is given by:
\beq\label{eq:svm_primal_scp}
\begin{aligned}
    \min \quad & \big((\mmzero; 1)\cproduct(D; D)\big) \bullet \mmx' \\
    \subto \quad & \underbrace{(-\mmP_j^T; \mmQ_j^T)}_{\mma_j} \bullet \mmx + \underbrace{\big((\mme_j; 0)\cproduct(0; 0)\big)}_{\mma_j'} \bullet \mmx' = 0, \ j = 1,\ldots, d\\
    & \underbrace{(\mmone_{n_1}; \mmzero)}_{\mma_{d+1}} \bullet \mmx + \underbrace{\big((\mmzero; 0)\cproduct( 1; 0)\big)}_{\mma_{d+1}'} \bullet \mmx' = 1\\
    & \underbrace{(\mmzero; \mmone_{n_2})}_{\mma_{d+2}} \bullet \mmx + \underbrace{\big((\mmzero; 0)\cproduct(0; 1)\big)}_{\mma_{d+2}'} \bullet \mmx' = 1\\
    & \mmx \in \cc{K}, \ \mmx' \in \cc{K}'
\end{aligned}
\eeq
Notice that by combining the last two constraints, we obtain an implicit constraint
\[
    \mmone_{n}\bullet\mmx + \big((\mmzero; 0)\cproduct( 1; 1)\big) \bullet \mmx' = 2.
\]
Here $\mmone_n = \mme$ (the identity in $\cc{K}$) and $(\mmzero; 0) \cproduct (1;1)\in \cc{K}'$. Therefore, assumption~\eqref{assumption} is implicitly satisfied by~\eqref{eq:svm_dual_scp} and~\eqref{eq:svm_primal_scp}. Moreover, we have $R = 2$.

As we assume that the point sets are separable, 0 is a lower bound for the optimal margin. From Lemma~\eqref{lem:svm_bound}, we also obtain an upperbound $s_1^* + s_2^* \le 2D$. To solve the SVM problem, we search for the optimal margin in the range $(0, 2D]$.

Fix an error tolerance $\eps>0$ and a guess $\alpha\in (0, 2D]$ for the optimal value $\OPT$. Let $\xt{\mmp}{t}$ be the vector that parameterizes the hyperplane $\xt{\cc{H}}{t}$: it lies in the cone $\cc{K} =\bR^{n}_+$. Let $\xt{\mmmu}{t}$ and $\xt{\mmgamma}{t}$ be the components of $\xt{\mmp}{t}$, where $\xt{\mmmu}{t}\in\bR^{n_1}_+$ and $\xt{\mmgamma}{t}\in\bR^{n_2}_+$. In other words, $\xt{\mmp}{t} = (\xt{\mmmu}{t}; \xt{\mmgamma}{t})$. The convex program that the {\sf ORACLE} needs to solve to check whether $\xt{\cc{H}}{t}_+ \cap \cc{C} \cap \cc{S}_\alpha = \varnothing$ (recall~\eqref{eq:min_oracle}) is given by:
\[\begin{aligned}
    \max_{\mmw,s_1,s_2} \quad & (\mmP^T\mmw - s_1\mmone)\bullet \xt{\mmmu}{t} + ({-}\mmQ^T\mmw - s_2 \mmone)\bullet \xt{\mmgamma}{t}\\
    \subto\quad & s_1 + s_2 \ge \alpha, \quad \|\mmw\|_2 \le 1, \quad s_1, s_2 \le D,
\end{aligned}\]
which can be re-organized as:
\begin{equation}\label{eq:svm_oracle}
    \begin{aligned}
        \max_{\mmw, s_1, s_2} \quad & (\mmP\xt{\mmmu}{t} - \mmQ\xt{\mmgamma}{t})^T \mmw - \Tr(\xt{\mmmu}{t}) s_1 - \Tr(\xt{\mmgamma}{t}) s_2 \\
        \subto\quad  & s_1 + s_2 \ge \alpha, \quad \|\mmw\|_2 \le 1, \quad s_1, s_2 \le D.
    \end{aligned}
\end{equation}
This problem has the following analytical solution:
\[
    \xt{\mmw}{t} = \frac{\mmP\xt{\mmmu}{t} - \mmQ\xt{\mmgamma}{t}}{\|\mmP\xt{\mmmu}{t} - \mmQ\xt{\mmgamma}{t}\|_2},\quad
    \xt{s}{t}_1 = \begin{cases}
        D          & \text{if $\Tr(\xt{\mmmu}{t}) \le \Tr(\xt{\mmgamma}{t})$}, \\
        \alpha - D & \text{otherwise},
    \end{cases}\quad
    \xt{s}{t}_2 = \alpha - \xt{s}{t}_1.
\]

\begin{claim} 
Let $\xt{\mmp}{t}$ be the vector that parameterizes the separating hyperplane $\xt{\cc{H}}{t}$ in the $t$-th iteration, and let $\xt{\mmmu}{t}, \xt{\mmgamma}{t}$ be its components.
If the {\sf ORACLE} fails to return a point in $\xt{\cc{H}}{t}_+ \cap \cc{C} \cap \cc{S}_\alpha$,
then $(\frac{\xt{\mmmu}{t}}{\Tr(\xt{\mmmu}{t})}; \frac{\xt{\mmgamma}{t}}{\Tr(\xt{\mmgamma}{t})})$ is a solution of~\eqref{eq:pd} with objective value at most $\alpha$.
\end{claim}

We give two distinct proofs of the above claim:
\begin{itemize}
  \item The {\sf ORACLE} fails when the optimal value of~\eqref{eq:svm_oracle} is negative, i.e.
        \begin{equation}\label{eq:pd_separation}
            (\mmP\xt{\mmmu}{t})^T\mmw - \Tr(\xt{\mmmu}{t})\xt{s}{t}_1 < (\mmQ\xt{\mmgamma}{t})^T\mmw + \Tr(\xt{\mmgamma}{t})\xt{s}{t}_2 \quad \text{for all }\|\mmw\|_2 \le 1.
        \end{equation}
        If $\Tr(\xt{\mmmu}{t})\le \Tr(\xt{\mmgamma}{t})$, we have $\xt{s}{t}_1 = D$ and $\xt{s}{t}_2 = \alpha - D$. For the left-hand side of~\eqref{eq:pd_separation}, we have:
        \[
            \max_{\mmw : \|\mmw\|_2\le 1}(\mmP\xt{\mmmu}{t})^T \mmw - \Tr(\xt{\mmmu}{t})\xt{s}{t}_1 \le \max_{\mmu \in \cc{P}} \Tr(\xt{\mmmu}{t})\|\mmu\|_2 - \Tr(\xt{\mmmu}{t})D \le 0.
        \]
        Therefore, {if we} scale the left-hand and right-hand side of~\eqref{eq:pd_separation} by $\frac{1}{\Tr(\xt{\mmmu}{t})}$ and $\frac{1}{\Tr(\xt{\mmgamma}{t})}$ respectively
        the inequality still holds, 
        {and}
        we have:
        \[
            (\textstyle\frac{\mmP\xt{\mmmu}{t}}{\Tr(\xt{\mmmu}{t})})^T\mmw - \xt{s}{t}_1 <
            (\textstyle\frac{\mmQ\xt{\mmgamma}{t}}{\Tr(\xt{\mmgamma}{t})})^T \mmw + \xt{s}{t}_2 \quad \text{for all }\|\mmw\|_2 \le 1.
        \]
        Re-arranging the inequality gives:
        \[
            \left\|\frac{\mmP\xt{\mmmu}{t}}{\Tr(\xt{\mmmu}{t})} - \frac{\mmQ\xt{\mmgamma}{t}}{\Tr(\xt{\mmgamma}{t})}\right\|_2 < \xt{s}{t}_1 + \xt{s}{t}_2 = \alpha.
        \]
        Since ${\Tr(}\frac{\xt{\mmmu}{t}}{\Tr(\xt{\mmmu}{t})}{)} = {\Tr(}\frac{\xt{\mmgamma}{t}}{\Tr(\xt{\mmgamma}{t})}{)} = 1$, we conclude that $(\frac{\xt{\mmmu}{t}}{\Tr(\xt{\mmmu}{t})}; \frac{\xt{\mmgamma}{t}}{\Tr(\xt{\mmgamma}{t})})$ is a feasible solution to~\eqref{eq:pd} with objective value at most $\alpha$.

        One can get the same result using the same technique for the $\Tr(\xt{\mmmu}{t}) > \Tr(\xt{\mmgamma}{t})$~case.

  \item Since the optimal value of~\eqref{eq:svm_oracle} is negative, it is equivalent to {saying} that the optimal value of the following problem is less than $\alpha$:
    \begin{equation}\label{eq:pd_relaxed}
    \begin{aligned}
        \max_{\mmw, s_1, s_2} \quad & s_1 + s_2                                                                                                    \\
        \subto \quad                & (\mmP\xt{\mmmu}{t} - \mmQ\xt{\mmgamma}{t})^T \mmw - \Tr(\xt{\mmmu}{t}) s_1 - \Tr(\xt{\mmgamma}{t}) s_2 \ge 0 \\
                                    & \|\mmw\|_2 \le 1, \quad s_1, s_2 \le D.
    \end{aligned}
    \end{equation}
    Observe that any point $(\mmw; s_1; s_2)$ {that} satisfies the constraints
    \begin{equation}\label{eq:pd_restricted}
    \begin{aligned}
        &(\mmP\xt{\mmmu}{t})^T\mmw - \Tr(\xt{\mmmu}{t})s_1 \ge 0,\\
        &-(\mmQ\xt{\mmgamma}{t})^T\mmw - \Tr(\xt{\mmgamma}{t})s_2 \ge 0,\\
        &\|\mmw\|_2 \le 1,
    \end{aligned}
    \end{equation}
    also {satisfies} the constraints in~\eqref{eq:pd_relaxed}.
    Thus the constraints in~\eqref{eq:pd_restricted} {are} more {restrictive} than~\eqref{eq:pd_relaxed}, and the optimal value of the following problem is also less than $\alpha$:
    \begin{equation}\label{eq:pd_dual_construction}
    \begin{aligned}
        \max_{\mmw, s_1, s_2} \quad & s_1 + s_2\\
        \subto\quad &(\mmP\xt{\mmmu}{t})^T\mmw - \Tr(\xt{\mmmu}{t})s_1 \ge 0,\\
        &-(\mmQ\xt{\mmgamma}{t})^T\mmw - \Tr(\xt{\mmgamma}{t})s_2 \ge 0,\\
        &\|\mmw\|_2 \le 1.
    \end{aligned}
    \end{equation}
    The dual problem of~\eqref{eq:pd_dual_construction} is:
    \begin{equation}\label{eq:pd_primal_construction}
    \begin{aligned}
        \min_{z_1, z_2, \mmxi, \xi_0}\quad & \xi_0\\
        \subto\quad & {-}(\mmP\xt{\mmmu}{t})z_1 {+} (\mmQ\xt{\mmgamma}{t})z_2 - \mmxi = \mmzero,\\
        & \Tr(\xt{\mmmu}{t})z_1 = \Tr(\xt{\mmgamma}{t})z_2 = 1,\\
        & {z_1, z_2 \ge 0, \quad}\|\mmxi\|_2 \le \xi_0.
    \end{aligned}
    \end{equation}
    Let $(\tilde{z}_1; \tilde{z}_2; \tilde{\mmxi}; \tilde{\xi}_0)$ be the optimal solution to the above problem.
    It is not difficult to see that $\tilde{z}_1 = \frac{1}{\Tr(\xt{\mmmu}{t})}$ and $\tilde{z}_2 = \frac{1}{\Tr(\xt{\mmgamma}{t})}$. Since the optimal value is less than $\alpha$, we have:
    \[
        \|\tilde{\mmxi}\|_2 = \left\|{-}(\mmP\xt{\mmmu}{t})\tilde{z}_1 {+} (\mmQ\xt{\mmgamma}{t})\tilde{z}_2\right\|_2 \le \tilde{\xi}_0 < \alpha.
    \]
    As $\Tr(\tilde{z}_1\xt{\mmmu}{t}) = \Tr(\tilde{z}_2\xt{\mmgamma}{t}) = 1$, the above result implies that $(\tilde{z}_1\xt{\mmmu}{t}; \tilde{z}_2\xt{\mmgamma}{t})$ is a feasible solution to~\eqref{eq:pd} with objective value at most $\alpha$.
\end{itemize}

Next, we show that the width $\rho$ of the {\sf ORACLE} is at most 2D. As explained above, at each iteration the {\sf ORACLE} returns a point $(\mmw; s_1; s_2)$ that satisfies
\beq\label{eq:svm_oracle_constr_1}
    \|\mmw\|_2 \le 1, \ s_1 \le D, \ s_2 \le D, \ s_1 + s_2 \ge \alpha.
\eeq
The last constraint, together with the upper bounds of $s_1$ and $s_2$ implies
\beq\label{eq:svm_oracle_constr_2}
    s_1 \ge \alpha - s_2 \ge -D \text{ and } s_2 \ge \alpha - s_1 \ge -D.
\eeq
The width then equals the maximum value of
\[
    \Big|\lambda_k\big(\mmP^T\mmw - s_1 \mmone; -\mmQ^T \mmw - s_2 \mmone \big) \Big|,
\]
over all $k\in [r]$ and $(\mmw; s_1; s_2)$ that satisfies~\eqref{eq:svm_oracle_constr_1} and~\eqref{eq:svm_oracle_constr_2}, which is
\[
    \max\Big\{ \max_{\mmu\in\cc{P}} \big|\mmu^T \mmw - s_1\big|, \
    \max_{\mmv\in\cc{Q}} \big|-\mmv^T \mmw - s_2 \big| \Big\}.
\]
Finally, we have
\[\begin{aligned}
    & \max_{\mmu\in \cc{P}}\big|\mmu^T \mmw - s_1\big| \overset{(a)}{\le} \max_{\mmu \in \cc{P}} \big|\mmu^T \mmw\big| + |s_1| \overset{(b)}{\le} D + |s_1| \overset{(c)}{\le} 2D,\\
    \text{and}\ & \max_{\mmv\in \cc{Q}}\big|-\mmv^T \mmw - s_2\big| \overset{(a)}{\le} \max_{\mmv \in \cc{Q}} \big|\mmv^T \mmw\big| + |s_2| \overset{(b)}{\le} D + |s_2| \overset{(c)}{\le} 2D,
\end{aligned}\]
where $(a)$ are due to the triangle inequality, $(b)$ follow from~\eqref{eq:svm_oracle_constr_1} and the definition of $D$, and $(c)$ are due to the bounds on $s_1$ and $s_2$ in~\eqref{eq:svm_oracle_constr_1} and~\eqref{eq:svm_oracle_constr_2}.

By Theorem~\ref{thm:min_meta} we know that if we execute the primal-dual meta-algorithm for $T = \left\lceil \frac{4 R^2 \rho^2 \ln r}{\eps^2 \alpha^2} \right\rceil$ iterations, we either correctly conclude infeasibility or we find a solution with value $(1-\eps)\alpha$. For the latter case, a hyperplane with normal vector $\bar{\mmw} = {1\over T}\sum_t \xt{\mmw}{t}$ can separate the point sets $\cc{P}$ and $\cc{Q}$ with a margin at least $(1-\eps)\alpha$. As discussed above, for problem~\eqref{eq:svm_new_form}, we have $R = 2, \rho\le 2D$ and $r = n$. The number of iterations $T$ can be simplified to $\left\lceil \frac{64 D^2\ln n}{\eps^2\alpha^2} \right\rceil$.

We now analyze the overall running time of the algorithm for solving the optimization problem.
Since we obtained lower and upper bounds for $\OPT$, where $L_0 = 0$ and $U_0 = 2D$. The technique that reduces optimization problems to $\alpha$-FTP introduced in Section~\ref{sec:algorithm} applies. The total number of iterations required for solving the optimization problem is $O(\frac{D^2\log n}{\eps^2\OPT^2})$. Define $E = \frac{D^2}{\OPT^2}$ as a instance-specific parameter. The iteration complexity can be simplified to $O(\frac{E\log n}{\eps^2})$.
The computation in each iteration of our algorithm includes:
$(a)$ Solving~\eqref{eq:svm_oracle} for test of separation. As in our previous discussion, the optimal solution to~\eqref{eq:svm_oracle} has an analytical solution. Computing the optimal solution and its objective value takes $O(nd)$ time.
$(b)$ Performing a {SCMWU} step for adjusting the hyperplane. This involves computing the loss vector, and performing the update and the normalization step. Since the vector $\xt{\mmp}{t}$ lies in $\bR^n_+$, this also requires $O(nd)$ time. 
Combining the iteration complexity and the running time for each iteration, we conclude that the total time complexity of our algorithm is $O(\frac{E \, nd \log n}{\eps^2})$.
\end{proof}

\begin{theorem}
The algorithm described in Theorem~\ref{thm:svm} can be parallelized under the work-span model.
The parallel algorithm takes $O\Big(\frac{End \log n}{\eps^2}\Big)$ work and has $O\Big(\frac{E\log(nd)\log n}{\eps^2}\Big)$ span.
\end{theorem}

\begin{proof}
Similar to Theorem~\ref{thm:ses_parallel}, we parallelize the computation during each iteration of the algorithm for the SVM/PD problem.
This involves computing the analytical solution of~\eqref{eq:svm_oracle} and adjusting the hyperplane using SCMWU for an element in $\bR^n_+$. Under the work-span model, the work of the parallel algorithm is the same as the sequential time complexity, while the span of each iteration becomes $O(\log (nd))$. When the searching range is small and a solution of~\ref{eq:pd} is required, we only need to normalize the vectors $\xt{\mmmu}{t}$ and $\xt{\mmgamma}{t}$, which can be done in $O(n)$ work and $O(\log n)$ span. Therefore, the total work of the parallel algorithm is $O\Big(\frac{End \log n}{\eps^2}\Big)$ and the span is $O\Big(\frac{E\log(nd)\log n}{\eps^2}\Big)$.
\end{proof}

\section{Experimental results}\label{sec:experiment}

{\bf Implementation strategies.}
We implemented our algorithms for the SES and the SVM/PD problem in both sequential (CPU) and parallel (GPU) settings.
{The CPU versions are developed using C++ and denoted as PDSCP-ST, while the GPU versions use NVIDIA's CUDA programming framework~\cite{cuda} and are denoted as~PDSCP. }
For both CPU and GPU versions, we employ the same implementation strategies.
We examine the solution we obtained in each iteration. For the SES problem, letting $\xt{\bar{\mmu}}{t}$ be the mean of the first $t$ iterates, we compute
\[
    \xt{f}{t} \triangleq \max_{i\in [n]} \|\xt{\bar{\mmu}}{t} - \mmv_i\|_2 + \gamma_i.
\]
For the SVM problem, letting $\xt{\bar{\mmw}}{t}$ be the first $t$ iterates, we compute
\[
    \xt{f}{t} \triangleq \max\Big\{\textstyle\frac{1}{\|\xt{\bar{\mmw}}{t}\|_2} \displaystyle\big(\min_{\mmv\in \cc{P}} \mmv^T\xt{\bar{\mmw}}{t} - \max_{\mmu\in \cc{Q}} \mmu^T\xt{\bar{\mmw}}{t}  \big),\ 0\Big\}.
\]
If the current result is already better than our guess $\alpha$ (for SES, when the current radius $\xt{f}{t}$ is smaller than $\alpha$; and for SVM, when the current margin $\xt{f}{t}$ is larger than $\alpha$), we can terminate the process and get an affirmative conclusion for the $\alpha$-feasibility test problem. Moreover, instead of running the algorithms for a sufficiently large number of iterations as stated in their theoretical analyses, we use an early-stopping technique that is often used in practice,~and terminate the process when the result is stabilized. Specifically, in each iteration, we check~whether
\[
    \frac{|\xt{f}{t} - \xt{f}{t-1}|}{\xt{f}{t-1}} < \delta,
\]
where $\delta$ is a small constant (in our experiments, $\delta = 10^{-4}$).\footnote{Note that when $\xt{f}{t-1} = 0$, we don't compute the ratio and the algorithm continues.}
If the criterion has been satisfied consecutively (in our experiments, 10 consecutive iterations), the algorithm will be terminated.

{\bf Experimental setups.}
We compare the performance of our algorithms 
{against the two popular commercial softwares Gurobi Optimizer~\cite{gurobi} and IBM ILOG Cplex Optimizer~\cite{cplex}, as well as the computational geometry algorithms library CGAL~\cite{cgal}.}
The main results of Gurobi and Cplex are executed with the default settings of using the barrier method on up to 8 threads.
For simplicity, the input instances of the SES problem are randomly generated point sets (i.e.,~$\gamma_i = 0$), in which case the problem can also be formulated as the following quadratic program~\cite{ses_exact:gartner2000efficient}:
\beq\begin{aligned}\label{eq:ses_qp}
    \min_{\mmx, \mmy} \quad & \mmy^T \mmy - \sum_{i=1}^n \mmv_i^T \mmv_i x_i\\
    \subto \quad & \mmy = \sum_{i=1}^n \mmv_i x_i,\\
    & \mmone^T \mmx = 1,\\
    & \mmx \ge \mmzero.
\end{aligned}\eeq

As for the SVM/PD problem, the input instances are generated to have similar parameters $E$.
The problem also has a quadratic programming formulation:
\beq\label{eq:svm_qp}
\begin{aligned}
    \min_{\mmmu,\mmgamma} \quad & \mmmu^T\mmP^T\mmP\mmmu + \mmgamma^T \mmQ^T \mmQ \mmgamma - 2\mmmu^T\mmP^T\mmQ\mmgamma\\
    \subto\quad & \mmone^T \mmmu = 1,\ \mmone^T \mmgamma = 1,\\
    & \mmmu, \mmgamma \ge \mmzero.
\end{aligned}
\eeq

We test the performance of Gurobi and Cplex for solving the SES and SVM/PD problems using both the SOCP formulations (\eqref{eq:ses} and~\eqref{eq:svm}, marked with ``-SOCP'' in the figures) and the quadratic program formulations (\eqref{eq:ses_qp} and~\eqref{eq:svm_qp}, marked with ``-QP'' in the figures). 
To better understand
{the parallel speedup they get from running on 8 threads,}
we also include results on them running 
{on}
just a single thread. The ``-ST'' suffix is used to denote the single-thread variants of the solvers. 
The algorithms implemented in CGAL are sequential and simplex-based methods for quadratic programming.
On the other hand, 
we did not manage to find any publicly available software adapting the core-set methods 
{for inclusion}
in our studies.
For the SVM problem, almost all the other existing solvers are designed for the soft-margin variants, which we do not consider in this work. 
All the experiments are executed on a machine with an Intel Core i7-9700K CPU with 32GB memory and 8 cores, and an NVIDIA RTX 2080Ti GPU with 11GB memory and 4352 CUDA cores. 
Each resulting running time reported here is averaged over 10 randomly generated instances. 
We terminate any experiment that takes beyond 1000 seconds and do not include them in our figures.

{\bf Different input sizes.}
The left chart in Figure~\ref{fig:ses_varing_number} and Figure~\ref{fig:svm_varying_number} shows the running time of different solvers for solving the SES problem and the SVM problem respectively, while the input size $n$ (for SVM, $n_1$ and $n_2$) ranges from $2^{10}$ to $2^{20}$. 
In both charts, the running times are illustrated in a logarithmic scale, and the number of dimensions $d$ is fixed to $64$.
The right chart in Figure~\ref{fig:ses_varing_number} and Figure~\ref{fig:svm_varying_number} shows the ratio of the running time of the best competitive solver to that of PDSCP for the SES and SVM problems.
The results demonstrate that our GPU implementation PDSCP significantly outperforms the multi-threaded SOCP solvers Gurobi-SOCP and Cplex-SOCP, which use the same problem formulations.
With a slower growth rate to the problem size, PDSCP also outperforms the QP solvers Gurobi-QP and Cplex-QP (which are better than their corresponding SOCP solvers) as the input size increases.
This is evidenced in the right chart of Figure~\ref{fig:ses_varing_number} and Figure~\ref{fig:svm_varying_number}, where PDSCP runs faster than the fastest QP solver once the input size is larger than $2^{17}$ for the SES problem and $2^{14}$ for the SVM problem in our experiments.
The growth rate of the running time (w.r.t.~the increase in input size) of our sequential implementation PDSCP-ST is roughly the same as the sequential interior point solvers, which is reasonable as in the theoretical analyses our algorithms have a dependence of $O(n\log n)$ on the input size and the interior point methods have a dependence of $O(n^{1.5})$.
The results of CGAL do not show up in Figure~\ref{fig:ses_varing_number} since they exceeded the time limit of 1000 seconds.

\begin{figure}[t]
\includegraphics[width=\linewidth]{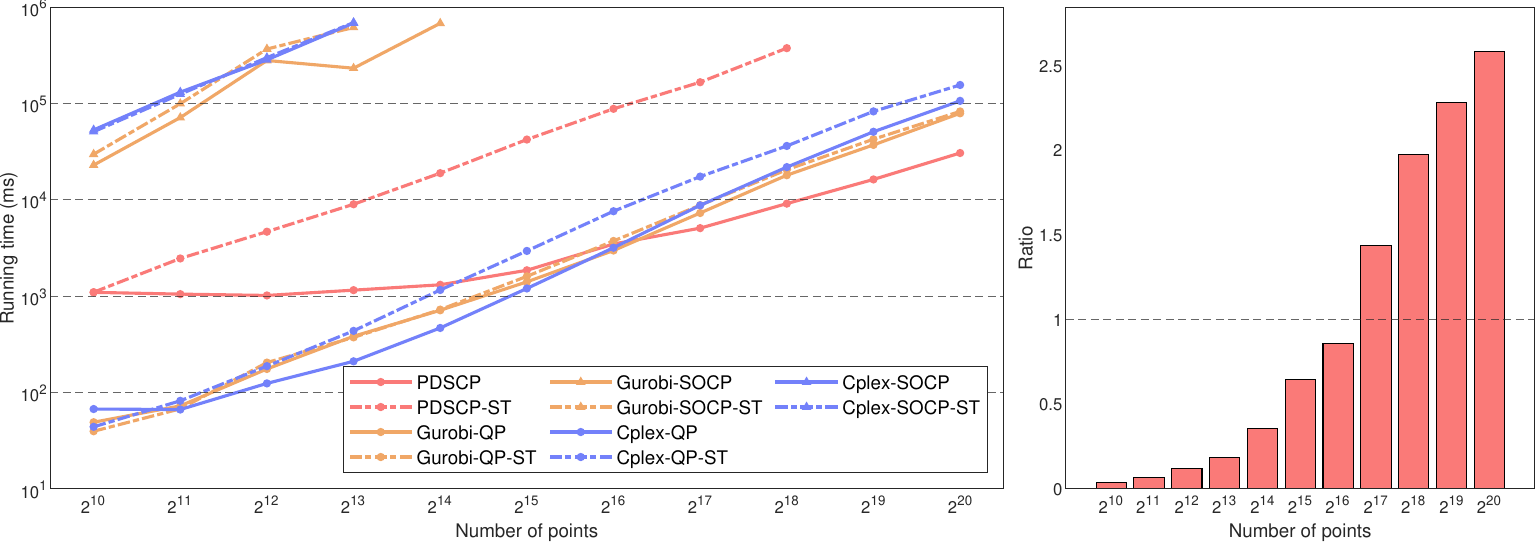}
\vspace*{-8mm}
\caption{(Left) Running time of different solvers with inputs of different sizes, and (Right) the ratio of the running time of the best competitive solvers to that of PDSCP for the {SES} problem.
}\label{fig:ses_varing_number}
\vspace{.5em}
\includegraphics[width=\linewidth]{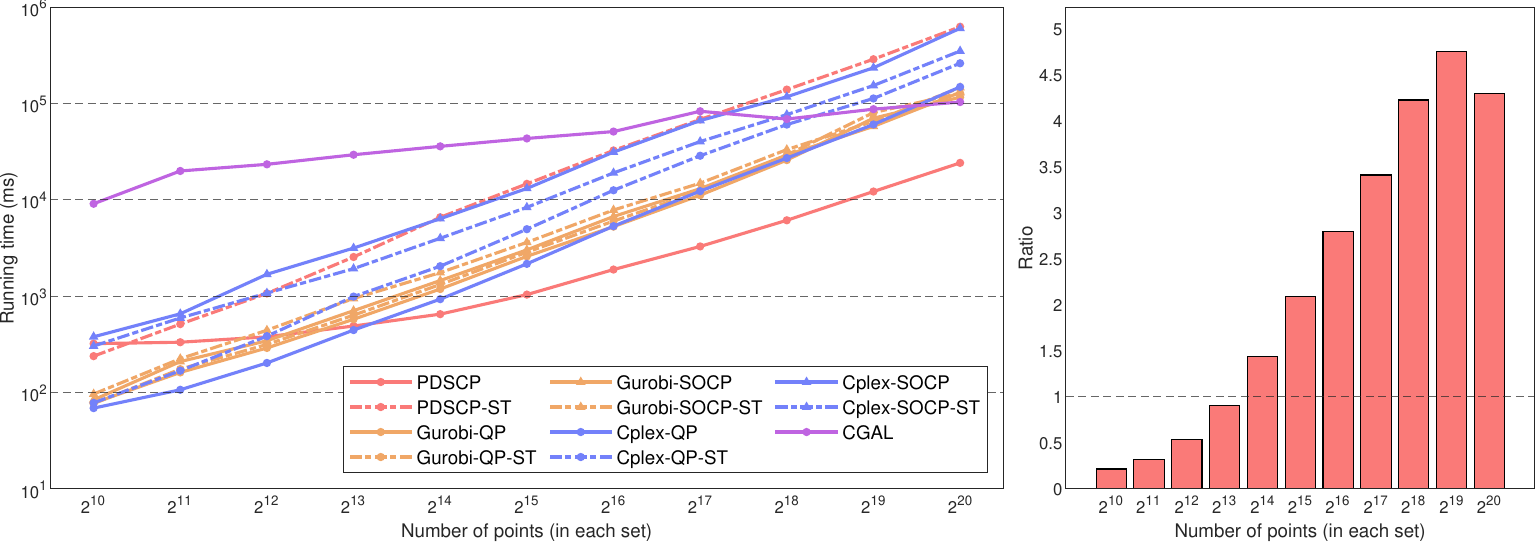}
\vspace*{-8mm}
\caption{ (Left) Running time of different solvers with inputs of different sizes, and (Right) the ratio of the running time of the best competitive solvers to that of PDSCP for the {SVM}~problem.}\label{fig:svm_varying_number}
\end{figure}

{\bf Different dimensionalities.}
The left charts in Figure~\ref{fig:ses_varing_dimension} and Figure~\ref{fig:svm_varying_dimension} show the running time of the solvers for solving the SES problem and the SVM problem respectively, with the dimensionality $d$ ranging from 2 to 512. In both figures, the running time is demonstrated on a logarithmic scale. For the SES problem, the number of input points $n$ is fixed to 100,000, and for the SVM problem, the number of points in each set ($n_1$ and $n_2$) is fixed to 100,000.
The right charts in Figure~\ref{fig:ses_varing_dimension} and Figure~\ref{fig:svm_varying_dimension} show the ratio of the running time of the best competitive solver to that of PDSCP for the SES and SVM problems.
The results demonstrate that PDSCP outperforms the multi-threaded SOCP solvers Gurobi-SOCP and Cplex-SOCP in most cases, which use the same problem formulations.
It also surpasses the QP solvers Gurobi-QP and Cplex-QP as the dimensionality increases. This is evidenced in the right chart of Figure~\ref{fig:ses_varing_dimension} and Figure~\ref{fig:svm_varying_dimension}, where PDSCP runs faster than the best QP solver once the dimensionality is larger than 128 for the SES problem and 16 for the SVM problem.
It is worth noting that PDSCP and PDSCP-ST have a significantly slower growth rate concerning the increase in dimensionality as compared with the other solvers (either the simplex-based solver of CGAL or interior point solvers of Gurobi and Cplex), which reflects the results we obtained from the theoretical analyses, where the running time of our algorithms have a $O(d)$ dependence on the number of dimensions and the dependences of simplex and interior point methods are super-quadratic.
It is reasonable to anticipate that the advantage of PDCSP will persist in higher dimensions.
The results for Gurobi-SOCP are not available because it runs out of memory during execution.

\begin{figure}[t]
\includegraphics[width=\linewidth]{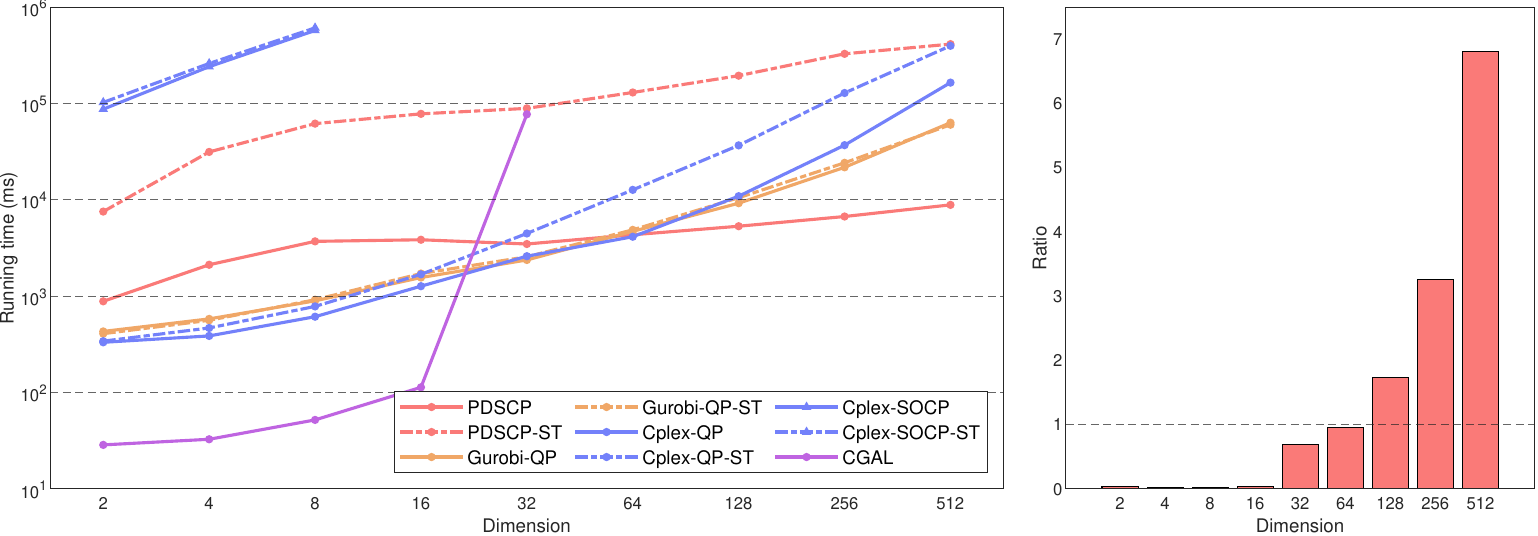}
\vspace*{-8mm}
\caption{(Left) Running time of the solvers under different dimensionalities, and (Right) the ratio of the running time of the best competitive solvers to that of PDSCP for the SES problem.}\label{fig:ses_varing_dimension}
\vspace{.5em}
\includegraphics[width=\linewidth]{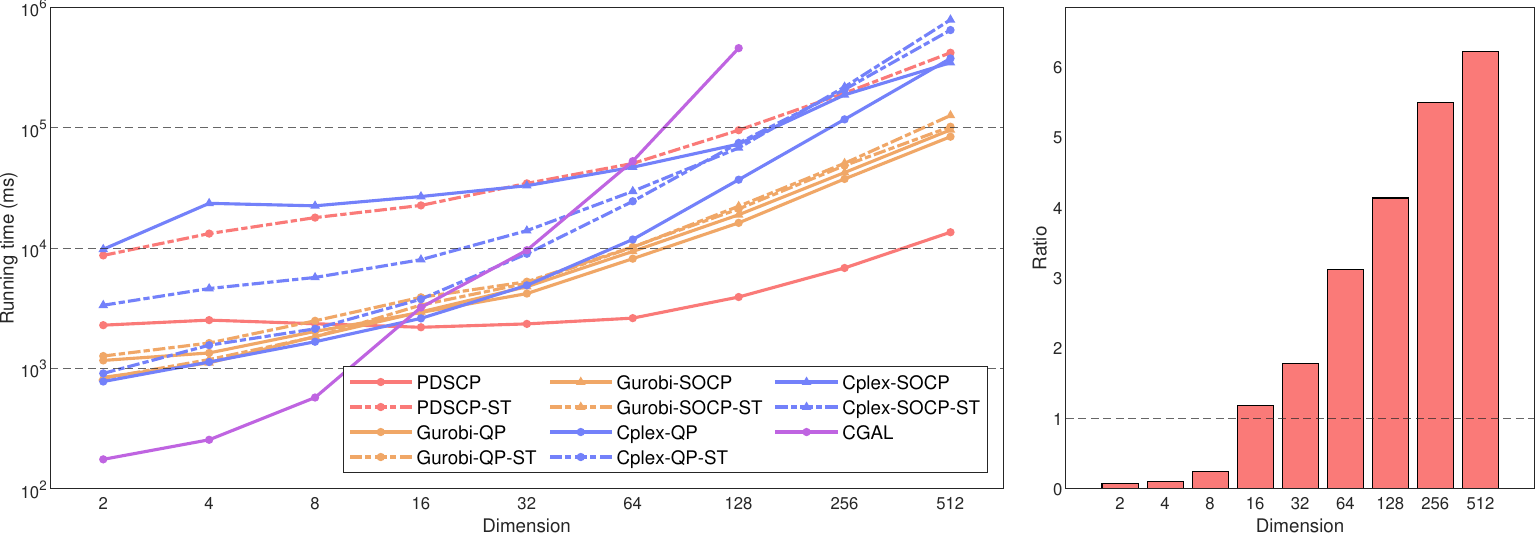}
\vspace*{-8mm}
\caption{(Left) Running time of the solvers under different dimensionalities, and (Right) the ratio of the running time of the best competitive solvers to that of PDSCP for the SVM problem.}\label{fig:svm_varying_dimension}
\end{figure}

{\bf Parallel speedup.}
The parallelism of a parallel solver is to be understood through its speedup, which is defined as the ratio of the running time of its single-thread version to the running time of its parallel version. From the left charts in Figure~\ref{fig:ses_varing_number}--\ref{fig:svm_varying_dimension}, it is evident that PDSCP demonstrates a significantly superior capability, achieving speeds over an order of magnitude faster than its sequential counterpart, PDSCP-ST.
On the other hand, the charts also depict that all the interior point solvers from Gurobi and Cplex do not seem to gain much benefit from parallelization.
The performances of their multi-thread versions are sometimes even worse than the sequential versions due to the extra overhead of multithreading. For example, the performance of Cplex-QP v.s.~Cplex-QP-ST in all the figures, and the performance of Gurobi-SOCP v.s.~Gurobi-SOCP-ST in Figure~\ref{fig:ses_varing_number} and Figure~\ref{fig:svm_varying_dimension}.

{\bf Quality of results.}
To measure the quality of the results of our algorithms, we compare them with the best result obtained from other solvers. Let $\hat{f}$ be the result obtained from PDSCP (for SES, the smallest $\xt{f}{t}$; and for SVM, the largest $\xt{f}{t}$), and let $f^*$ be the best result obtained from other solvers. For the SES problem, we define the error as $(\hat{f} - f^*) / f^*$; and for the SVM problem, we define the error as $(f^* - \hat{f}) / f^*$. Notice that the errors can have negative values, in which cases our results are better than the best results of the other solvers.
The average errors of PDSCP for the SES problem are given in Table~\ref{table:ses_error}, and the average errors for the SVM problem are shown in Table~\ref{table:svm_error}.
As we can see from the tables, the error exhibits a slow rate of growth as the input size increases.
The scenario concerning the increase of dimensionality differs. For the SES problem, the error increases as the dimensionality grows, but for the SVM problem, the error decreases.
In general, PDSCP produces a desirable result for most problem instances.
For the SES problem, the 80th percentile of the errors over all experimented cases is 0.0056; and for the SVM problem, the 80th percentile of the errors over all inputs is 0.0008.

\begin{table*}[t]
\centering
\caption{Average errors of PDSCP for the SES problem}
\resizebox{\textwidth}{!}{
\begin{tabular}{|c | ccccccccccc|}
\specialrule{1pt}{0.3em}{0.05em}
Number of points & $2^{10}$ & $2^{11}$ & $2^{12}$ & $2^{13}$ & $2^{14}$ & $2^{15}$ & $2^{16}$ & $2^{17}$ & $2^{18}$ & $2^{19}$ & $2^{20}$\\
\hline
Average error & 0.0019 & 0.0021 & 0.0023 & 0.0024 & 0.0025 & 0.0029 & 0.0031 & 0.0042 & 0.0041 & 0.0044 & 0.0055\\ 
\specialrule{1pt}{0.05em}{0.2em}
\end{tabular}
}
\resizebox{0.9\textwidth}{!}{
\begin{tabular}{|c | ccccccccc|}
\specialrule{1pt}{0em}{0.05em}
Number of dimensions & 2 & 4 & 8 & 16 & 32 & 64 & 128 & 256 & 512\\
\hline
Average error & 0.0003 & 0.0009 & 0.0008 & 0.0012 & 0.0024 & 0.0027 & 0.0058 & 0.0085 & 0.0098
\\ 
\specialrule{1pt}{0.05em}{0em}
\end{tabular}
}
\label{table:ses_error}
\end{table*}

\begin{table*}[t]
\centering
\caption{Average errors of PDSCP for the SVM problem}
\resizebox{\textwidth}{!}{
\begin{tabular}{|c | ccccccccccc|}
\specialrule{1pt}{0.3em}{0.05em}
Number of points & $2^{10}$ & $2^{11}$ & $2^{12}$ & $2^{13}$ & $2^{14}$ & $2^{15}$ & $2^{16}$ & $2^{17}$ & $2^{18}$ & $2^{19}$ & $2^{20}$\\
\hline
Average error & 0.0004 & 0.0004 & 0.0005 & 0.0005 & 0.0006 & 0.0002 & 0.0006 & 0.0006 & 0.0007 & 0.0007 & 0.0007
\\ 
\specialrule{1pt}{0.05em}{0.2em}
\end{tabular}
}
\resizebox{0.9\textwidth}{!}{
\begin{tabular}{|c | ccccccccc|}
\specialrule{1pt}{0em}{0.05em}
Number of dimensions & 2 & 4 & 8 & 16 & 32 & 64 & 128 & 256 & 512\\
\hline
Average error & 0.0025&0.0023&0.0015&0.0011&0.0008&0.0006&0.0006&0.0006&0.0007
\\ 
\specialrule{1pt}{0.05em}{0em}
\end{tabular}
}
\label{table:svm_error}
\end{table*}

\section{Concluding remark}\label{sec:conclusion}
In this work, we introduce a novel primal-dual framework for symmetric cone programming (SCP)
{that utilizes}
a recent extension of the multiplicative weights update method to symmetric cones. 
We applied this framework to devise parallel algorithms for smallest enclosing sphere and support vector machine (polytope distance).
Our implementation of these two algorithms, tested in both sequential and parallel settings, demonstrated excellent efficiency and scalability with large-scale inputs.
Looking forward, the versatility of SCP allows our framework to have broad applications in both theoretical and practical aspects.
Besides the width-dependent method proposed in this paper, width-independent algorithms for special classes of SCPs (such as positive SCPs) would also be an interesting topic to~study.

\bibliography{bibtex}

\end{document}